\documentclass[english,11pt]{smfart}

\usepackage{etex}

\usepackage{amsbsy}
\usepackage{amsmath,amsfonts,amssymb,amsthm,mathrsfs,mathtools}

\usepackage{bm}

\usepackage[a4paper,vmargin={3cm,3cm},hmargin={3.5cm,3.5cm}]{geometry}
\usepackage[font=sf, labelfont={sf,bf}, margin=1cm]{caption}
\usepackage{graphicx}
\usepackage{epsfig}
\usepackage{latexsym}
\usepackage{xcolor}
\usepackage{ae,aecompl}
\usepackage{soul,framed}
\usepackage{comment}

\usepackage{xcolor}
\usepackage[pdfpagemode=UseNone,bookmarksopen=false,colorlinks=true,urlcolor=blue,citecolor=blue,citebordercolor=blue,linkcolor=blue]{hyperref}
\usepackage{smfhyperref}
\usepackage[capitalize]{cleveref}

\usepackage{pstricks}
\usepackage{enumerate}
\usepackage{tikz,animate,media9}						
\usepackage{todonotes}
\usepackage{pifont}
\usepackage{bm,marvosym}
\usepackage{algorithm}
\usepackage{algorithmic}

\usepackage{bbm}

\usepackage{tcolorbox}

\definecolor{aleacolor}{rgb}{0.16,0.59,0.78}

\hypersetup{
	breaklinks,
	colorlinks=true,
	linkcolor=aleacolor,
	urlcolor=aleacolor,
	citecolor=aleacolor}

\newcount\colveccount
\newcommand*\colvec[1]{
	\global\colveccount#1
	\begin{pmatrix}
		\colvecnext
	}
	\def\colvecnext#1{
		#1
		\global\advance\colveccount-1
		\ifnum\colveccount>0
		\\
		\expandafter\colvecnext
		\else
	\end{pmatrix}
	\fi
}

\newcommand{\ndN}{\mathbb{N}}

\newcommand{\ndR}{\mathbb{R}}
\newcommand{\ndC}{\mathbb{C}}

\renewcommand{\Pr}[1]{\mathbb{P}(#1)}

\newcommand{\Prb}[1]{\mathbb{P}\left(#1\right)}

\newcommand{\Ex}[1]{\mathbb{E}[#1]}

\newcommand{\Exb}[1]{\mathbb{E}\left[#1\right]}

\newcommand{\convdis}{\,{\buildrel d \over \longrightarrow}\,}
\newcommand{\convd}{\,{\buildrel d \over \longrightarrow}\,}

\newcommand{\convp}{\,{\buildrel p \over \longrightarrow}\,}

\newcommand{\eqdist}{\,{\buildrel d \over =}\,}

\newcommand{\cB}{\mathcal{B}}
\newcommand{\cC}{\mathcal{C}}
\newcommand{\cD}{\mathcal{D}}
\newcommand{\cE}{\mathcal{E}}
\newcommand{\cF}{\mathcal{F}}

\newcommand{\cH}{\mathcal{H}}
\newcommand{\cI}{\mathcal{I}}

\newcommand{\cL}{\mathcal{L}}
\newcommand{\cM}{\mathcal{M}}
\newcommand{\cN}{\mathcal{N}}

\newcommand{\cP}{\mathcal{P}}
\newcommand{\cQ}{\mathcal{Q}}

\newcommand{\cS}{\mathcal{S}}

\newcommand{\mC}{\mathsf{C}}
\newcommand{\mD}{\mathsf{D}}

\newcommand{\mG}{\mathsf{G}}
\newcommand{\mH}{\mathsf{H}}

\newcommand{\mM}{\mathsf{M}}

\newcommand{\mO}{\mathsf{O}}

\newcommand{\mQ}{\mathsf{Q}}

\newtheorem{theorem}{Theorem}[section]
\newtheorem{conjecture}[theorem]{Conjecture}
\newtheorem{corollary}[theorem]{Corollary}
\newtheorem{proposition}[theorem]{Proposition}
\newtheorem{lemma}[theorem]{Lemma}

\newtheorem{definition}[theorem]{Definition}

\numberwithin{equation}{section}

\keywords{cubic planar graphs, local convergence, core decompositions}

\title{\textbf{The Uniform Infinite Cubic Planar Graph}}
\date{}

\author{Benedikt Stufler}

\address[Benedikt Stufler]{Vienna University of Technology}
\email{benedikt.stufler at tuwien.ac.at}

\begin{document}

\vspace {-0.5cm}

\begin{abstract}
	We prove that the random simple connected cubic planar graph $\mathsf{C}_n$ with an even number $n$ of vertices admits a novel uniform infinite cubic planar graph (UICPG) as quenched local limit. We describe how the limit may be constructed by a series of random blow-up operations applied to the dual map of the type~III Uniform Infinite Planar Triangulation established by Angel and Schramm (Comm. Math. Phys., 2003).  Our main technical lemma is a contiguity relation between $\mathsf{C}_n$ and a model where the networks inserted at the links of the largest $3$-connected component of $\mathsf{C}_n$ are replaced by independent copies of a specific Boltzmann network. We prove that the number of vertices of the largest $3$-connected component concentrates at $\kappa n$ for $\kappa \approx 0.85085$, with Airy-type fluctuations of order $n^{2/3}$. The  second-largest component is shown to have significantly smaller size $O_p(n^{2/3})$.
\end{abstract}


\maketitle

\section{Introduction}

Cubic planar graphs are $3$-regular graph that admit a crossing-free embedding in the plane, or equivalently the $2$-sphere. Their study has  received increasing attention in recent  literature:  The asymptotic growth of the number of cubic planar graphs and multigraphs with $n$ vertices was determined in \cite{zbMATH05122852, zbMATH07213288}. Properties of random cubic planar graphs were studied in these works and  also in~\cite{zbMATH06639396,10.5565/PUBLMAT6612213}. The investigation of cubic planar graphs also stimulated further research directions, such as the study of $4$-regular planar graphs~\cite{zbMATH06827273} or cubic graphs on general orientable surfaces~\cite{zbMATH06841874}.

Throughout this work we let $\mC_n$ denote the simple connected cubic planar graph drawn uniformly at random from the collection of such graphs with a fixed $n$-element vertex set. It is clear that  this only makes sense when $n \ge 4$ is an even number, and that $\mC_n$ has $3n/2$ edges.

Our first main result determines the asymptotic local shape of $\mC_n$  near a uniformly selected vertex $v_n \in \mC_n$.

\begin{theorem}
	\label{te:main}
There is a Uniform Infinite Planar Cubic Graph (UICPG) $\hat{\mC}$ such that 
	\begin{align}
		\label{eq:annealed}
		(\mC_n, v_n) \convdis \hat{\mC}
	\end{align}
in the local topology. Furthermore, the conditional law $\mathfrak{L}((\mC_n, v_n) \mid \mC_n)$ admits the law $\mathfrak{L}(\hat{\mC})$ as distributional limit of random probability measures
\begin{align}
	\label{eq:quenched}
	\mathfrak{L}((\mC_n, v_n) \mid \mC_n) \convdis \mathfrak{L}(\hat{\mC}).
\end{align}
\end{theorem}
The \emph{annealed} convergence in~\eqref{eq:annealed} is a distributional limit of random elements of the space $(\mathfrak{G}, d_{\mathrm{loc}})$ of vertex-rooted locally finite connected simple graphs equipped with the local distance. In more concrete terms,~\eqref{eq:annealed} signifies that for any integer $k \ge 1$  the $k$-neighbourhood  $U_k(\mC_n, v_n)$ of the rooted graph $(\mC_n, v_n)$ satisfies for any rooted graph $H$
\[
	\Pr{U_k(\mC_n, v_n) \simeq H} \to \Pr{U_k(\hat{\mC}) \simeq H}.
\]
Here $\simeq$ denotes the existence of a root-preserving graph isomorphism. 

The conditional law $\mathfrak{L}((\mC_n, v_n) \mid \mC_n)$ corresponds to the uniform law of the $n$ vertex-rooted version of the random graph $\mC_n$.   The \emph{quenched} convergence in~\eqref{eq:quenched} is a distributional limit of random elements of the space $\mathbb{M}_1(\mathfrak{G})$ of probability measures on the Borel sigma-algebra of $\mathfrak{G}$.  In more concrete terms, the limit~\eqref{eq:quenched} means that the number $N_{H,k}$ of vertices in $\mC_n$ whose $k$-neighbourhood is isomorphic to $H$ satisfies
\[
	\frac{N_{H,k}}{n} \convp \Pr{U_k(\hat{\mC}) \simeq H}.
\]
Equivalently, quenched local convergence means that if we take two independent random vertices $v_n^{(1)}$ and $v_n^{(2)}$  of $\mC_n$ then the joint distributional convergence
\[
((\mC_n, v_n^{(1)}), (\mC_n, v_n^{(2)})) \convd (\hat{\mC}^{(1)},\hat{\mC}^{(2)})
\]
holds with $\hat{\mC}^{(1)}$ and $\hat{\mC}^{(2)}$ denoting independent copies of $\hat{\mC}$. The uniform $n$-vertex cubic planar graph that is not required to be connected admits the same random graph as local limit. See~Section~\ref{sec:cubicdisconnected} below for details on this extension.

The limit  UICPG is almost surely recurrent by the famous result~\cite{MR1873300} for locally convergent sequences of random graphs with bounded degrees, which was later generalized in~\cite{MR3010812} to graphs with light-tailed degree distributions. The name Uniform Infinite Planar Cubic Graph follows the naming tradition  of limits for different models of  random networks, as in the pioneering work on the Uniform Infinite Planar Triangulation~\cite{MR2013797}, and  further work in this active research field~\cite{MR3183575,MR3769811,budzinski2020local, MR3083919,zbMATH07306969, MR3256879,planar,kang2021local}.  It appears that still less is known about cubic planar graphs than about these models.

The structural results of~\cite{zbMATH05122852} showed that cubic planar graphs and networks may be decomposed recursively into series, parallel, isthmus, loop, and polyhedral networks. In particular, polyhedral networks consist of a $3$-connected cubic planar graph, whose edges may be replaced by non-isthmus networks. These $3$-connected graphs encountered in the full decomposition are called the $3$-connected component. We prove that the random cubic planar graph $\mC_n$ contains a unique giant $3$-connected component:

\begin{theorem}
	\label{te:main2}
	Let $V_n$ denote the number of vertices in the largest $3$-connected component of the uniform random $n$-vertex cubic planar graph $\mC_n$. Let
	\[
		h(t) = \frac{1}{\pi t} \sum_{n \ge 1} (-t 3^{2/3})^n \frac{\Gamma(2n/3 +1)}{n!} \sin(-2n\pi/3), \qquad t \in \ndR
	\]
	denote the density of the map type Airy distribution. There are algebraic constants \begin{align*}
			\kappa &= 0.850853090058314333870385348879612617197477\ldots \\
			c_v &=    1.205660773457703954344217302817493214574105\ldots 
	\end{align*}
	such that for any constant $M>0$
	\begin{align}
		\Pr{V_n = \kappa n + t n^{2/3}} = n^{-2/3} (2 c_v h(c_v t) + o(1))
	\end{align}
	uniformly for all $t \in [-M, M]$ satisfying $\kappa n + t n^{2/3} \in 2 \ndN$. Consequently, 
	\begin{align}
		\label{eq:vclt}
		\frac{V_n - \kappa n}{n^{2/3}} \convdis V_{3/2}
	\end{align}
	for a $3/2$-stable random variable $V_{3/2}$  with density $c_v h(c_v t)$.
\end{theorem}
The constants $\kappa$ and $c_v$ admit algebraic expressions given in Equations~\eqref{eq:kappa} and ~\eqref{eq:cvconst} below. Since $3$-connected components of cubic planar graphs do not overlap (see Section~\ref{sec:netdec} for more details), it follows directly from $\kappa>1/2$ that the largest component is with high probability unique. The second-largest $3$-connected component has size $O_p(n^{2/3})$, see Corollary~\ref{co:second} below.

Exhibiting linearly sized largest components is characteristic for planar structures~\cite{MR1871555,banderier2021phase}, and we employ a mix of these analytic methods and singularity expansions from~\cite{zbMATH05122852, zbMATH07213288}. Limit theorems for the largest $2$-connected and $3$-connected components in various classes of  cubic planar maps (as opposed to graphs) are given in the work~\cite{coreprep}, alongside new proofs for their enumeration.

We may view $\mC_n$ as the result of blowing up the $3V_n/2$ edges of its  $3$-connected core $\cM(\mC_n)$ by non-isthmus networks $(\cD_i(\mC_n))_{1 \le i \le 3 V_n/2}$. To be fully precise, $\cM(\mC_n)$ is only well-defined if $\mC_n$ has a unique largest $3$-connected component. However, since this happens with probability tending to $1$, we may safely set $\cM(\mC_n)$ to an arbitrary place-holder value if there is more than one $3$-connected component with maximal size.

 An important part  of the proof of Theorems~\ref{te:main} is that $\mC_n$ satisfies a contiguity relation to a model where these components are resampled independently according to a Boltzmann  network model $\mD$ defined in Definition~\ref{de:defd} below. We emphasize this result here, as we believe that it has many further uses for describing the asymptotic shape of the uniform cubic planar graph $\mC_n$, that go beyond the applications considered in the present work.

\begin{theorem}	
	\label{te:main3}
	Let $(\mD(i))_{i \ge 1}$ denote independent copies of the Boltzmann  network $\mD$.
	For any $\epsilon >0$ and  $0< \delta < 3 \kappa / 2$ there exist constants $0<c<C$ and $N >0$ and sets $(\cE_n)_{n \ge N}$ such that for all $n \in 2 \ndN$ with $n \ge N$ 
	\begin{align}
		\label{eq:doa}
		\Prb{ (\cM(\mC_n), (\cD_i(\mC_n))_{1\le i \le 3V_n/2 - \lfloor \delta n \rfloor })  \notin \cE_n  } < \epsilon
	\end{align}
	and
	\begin{align}
		\label{eq:dob}
		\Prb{ (\cM(\mC_n),  (\mD(i))_{1\le i \le 3V_n/2 - \lfloor \delta n \rfloor }) \notin \cE_n  } < \epsilon
	\end{align}
	and for all elements $E \in \cE_n $
	\begin{align}
		\label{eq:N2}
		c < \frac{\Prb{ (\cM(\mC_n), (\cD_i(\mC_n))_{1\le i \le 3V_n/2 - \lfloor \delta n \rfloor })  =E  }} {	\Prb{ (\cM(\mC_n),  (\mD(i))_{1\le i \le 3V_n/2 - \lfloor \delta n \rfloor }) =E  }}  < C.
	\end{align}
\end{theorem}

It appears that the assumptions in Theorem~\ref{te:main3} on the number $\lfloor \delta n \rfloor$ of components cannot be relaxed: As $\delta \to 0$, the lower bound $c$  we construct in the proof tends to zero and the upper bound $C$ tends to infinity. Moreover, the total variation distance between $(\cM(\mC_n), (\cD_i(\mC_n))_{1\le i \le 3V_n/2 - \lfloor \delta n \rfloor })$ and $(\cM(\mC_n),  (\mD(i))_{1\le i \le 3V_n/2 - \lfloor \delta n \rfloor })$ does \emph{not} tend to zero: Knowing the total mass of a linear number of components places a bias on the size of the core $\cM(\mC_n)$. See also the remarks at the end of Section~\ref{sec:components} below for further discussions of this.

The size of the Boltzmann network $\mD$ follows a power law with index $-5/2$, see Equation~\eqref{eq:dtail}. By extremal value statistics it follows immediately  that the largest element in $n$ independent samples of $\mD$ has size $O_p(n^{2/3})$.  Thus, Theorem~\ref{te:main3} and a time reversal argument immediately determine an upper bound for the size of the second-largest $3$-connected component of $\mC_n$:

\begin{corollary}
	\label{co:second}
	The number of vertices $|\cD_i(\mC_n)|$ satisfies
	\begin{align}
		\label{eq:uppera}
		\max_{1 \le i \le 3 V_n /2} |\cD_i(\mC_n)| = O_p(n^{2/3}).
	\end{align}
	For any sequence $t_n \to \infty$ the number of vertices $V_n^{(2)}$ of the second-largest $3$-connected component of $\mC_n$ satisfies
	\begin{align}
		\label{eq:lowerb}
		\Pr{ n^{2/3} / t_n \le V_n^{(2)} \le n^{2/3} t_n} \to 1
	\end{align}
as $n \in 2 \ndN$ tends to infinity.
\end{corollary}
The lower stochastic bound in~\eqref{eq:lowerb} requires a few extra arguments. We provide a proof at the end of Section~\ref{sec:components} below.

In our proof of Theorem~\ref{te:main} we also use Theorem~\ref{te:main3} to clarify the connection between the  Uniform Infinite Cubic Planar Graph and the (type III) Uniform Infinite Planar Triangulation constructed in~\cite{MR2013797}. Specifically, the UICPG may be constructed in three steps:
\begin{enumerate}
	\item Construct the dual map $\hat{\mM}$ of the type III UIPT.
	\item Replace each non-root edge of $\hat{\mM}$ by an independent copy of the Boltzmann network $\mD$ specified in Definition~\ref{de:defd}.
	\item Replace the root-edge of $\hat{\mM}$ by a size-biased version $\hat{\mD}$ of the Boltzmann network specified in Definition~\ref{def:sizebiased}.
\end{enumerate}
The edge replacement process is illustrated in  Figure~\ref{fi:poly}. The root vertex of the UICPG is determined by selecting an edge associated to $\hat{\mD}$ at random and distinguishing one of its ends according to a fair independent coin flip.


Specifically, we apply Theorem~\ref{te:main} to show that the shape of $\mC_n$ is determined by the $3$-connected core with the attached components resampled by independent copies of $\mD$. This is justified by arguing that the smaller we take the $\epsilon>0$ in Theorem~\ref{te:main3}, the less vertices will ``observe'' the $\epsilon n$ components that we leave untouched. That is, for any fixed  integer $k \ge 1$, taking $\epsilon>0$ small diminishes the percentage of vertices whose  graph-distance is at most $k$ from one of these $\epsilon n$ components. This allows us to deduce that local convergence of $\mC_n$ is equivalent to local convergence of a model where \emph{all} $3 V_n /2$ components attached to the links of the $3$-connected core of $\mC_n$ are replaced by independent copies of $\mD$. Now, in this model, a random point may either belong to the $3$-connected core, or to one of these components. If it falls into one of these components, then by the famous waiting time paradox the distribution of this component follows a \emph{size-biased} version of $\mD$, because larger components are more likely to be observed than smaller ones. The core itself is distributed like the dual of a randomly sized simple triangulation, and we argue that it converges in the quenched sense to the dual of the  type III Uniform Infinite Planar Triangulation. All remaining components observed by a random point behave like independent copies of $\mD$, hence summing over all possible partitions of the neighbourhood of a random point into parts from the core, parts from the size-biased component, and parts from the remaining components we obtain convergence of $\mC_n$ to the limit object described in the three steps above.

\subsection*{Summary of proof strategy}

The proof may be divided into two parts, carried out in Sections~\ref{sec:qloc} and~\ref{sec:cuplag}. We briefly summarize the steps for each part. 

\paragraph*{First part.} The breakthrough work~\cite{MR2013797} showed annealed local convergence of various types of random planar triangulations towards a Uniform Infinite Planar Triangulation of the corresponding type. This includes the model of so-called type III triangulations, which are $3$-connected and simple. In Section~\ref{eq:sipla} we extend the annealed limit to quenched local convergence. Section~\ref{sec:cupam} transfers this convergence by multiple applications of the continuous mapping theorem to the dual map construction. We obtain quenched local convergence of uniform random $3$-connected cubic planar maps. By Whitney's theorem, any $3$-connected cubic planar graph corresponds to precisely two $3$-connected planar maps. Hence  these random maps are distributed like uniform random $3$-connected cubic planar graphs.

\paragraph*{Second part.} Cubic networks are cubic planar graphs with an oriented root-edge that is allowed to be a loop or part of a double edge. Section~\ref{sec:netdec} recalls the decomposition of cubic networks by~\cite{zbMATH05122852} into parallel, series, loop, isthmus, and polyhedral networks. In particular, polyhedral networks are obtained from $3$-connected cubic planar graphs by replacing their edges with non-isthmus networks. Hence it consists of a $3$-connected core and non-isthmus components.  Section~\ref{sec:largest} determines the size of the  largest $3$-connected core in a random cubic planar graph. We provide a proof of Theorem~\ref{te:main2} that accurately quantifies the fluctuations around a constant multiple $\kappa n$. Thus, the shape of $\mC_n$ is determined by giant $3$-connected component whose edges are replaced by smaller networks.
In  Section~\ref{sec:components}  we study these smaller components attached to the giant core in more detail. We prove Theorem~\ref{te:main3},  showing that for any fixed but arbitrarily small $\epsilon>0$, the core together with all but the first $\epsilon n$ components satisfies a \emph{contiguity} relation to a modified model where these components are independent copies of a Boltzmann  network. The smaller the $\epsilon$ the smaller the impact on the local distance. This allows us to deduce in the final Section~\ref{sec:endgame} that quenched local convergence of large cubic planar networks is equivalent to a model where \emph{all} components attached to the core are independent copies of a Boltzmann network. Care has to be taken that, as in the famous waiting time paradox, a random vertex is more likely to fall in a large component than in a small one. Hence the vicinity of a random point  consists of parts in a \emph{size-biased} Boltzmann network, parts of the $3$-connected core, and parts of the remaining independent Boltzmann  networks that are attached to the core. Summing over all possible configurations of parts and using the convergence of large $3$-connected planar cubic graphs from Section~\ref{sec:qloc} it follows that both models admit a quenched local limit. 

\subsection*{Notation}
We  let $\mathbb{N} = \{1, 2, \ldots\}$ denote the collection of positive integers. We assume all considered random variables to be defined on a common probability space whose measure is denoted by $\mathbb{P}$. All unspecified limits are taken as $n$ becomes large, possibly taking only values in a subset of the natural numbers.  We use $\convdis$ and $\convp$ to denote convergence in distribution and probability. Equality in distribution is denoted by $\eqdist$. Weak convergence of probability measures is denoted by  $\Rightarrow$. An event holds with high probability, if its probability tends to $1$ as $n \to \infty$.
We let $O_p(1)$ denote an unspecified random variable $X_n$ of a stochastically bounded sequence $(X_n)_n$, and write $o_p(1)$ for a random variable $X_n$ with $X_n \convp 0$.  The total variation distance is denoted by $d_{\mathrm{TV}}$. The law of a random variable $X$ is denoted by $\mathfrak{L}(X)$.


\section{Quenched local convergence of $3$-connected cubic planar graphs}

\label{sec:qloc}

The goal of this section is to establish quenched local convergence of uniform $3$-connected cubic planar graphs. We proceed in two steps. First, we verify quenched convergence of simple planar triangulations to Angel and Schramm's Uniform Infinite Planar Triangulation. In the second step, we apply the continuous mapping theorem to transfer this convergence to the dual maps. The dual of a $3$-connected map is $3$-connected, so we arrive at local convergence of uniform $3$-connected cubic planar maps. By Whitney's theorem, each $3$-connected map has precisely two embeddings in the plane, hence this is equivalent to convergence of $3$-connected cubic planar graphs.

\subsection{Simple planar triangulations}
\label{eq:sipla}

Roughly speaking, a \emph{planar map} is a drawing of a connected multi-graph on the $2$-sphere such that edges are represented by arcs that may only intersect at their endpoints. Planar maps are considered up to orientation-preserving homeomorphism of the $2$-sphere, so that there only finitely many maps with a given number of edges. We will only consider rooted planar maps, where a \emph{root edge} is distinguished and oriented in order to eliminate symmetries.  The \emph{faces} of a planar map correspond to the connected components created when removing the planar map from the $2$-sphere. We refer to the face to the right of the root edge as the outer face. Each face has a \emph{boundary}, consisting of a counter-clockwise cyclically ordered list of sides of edges. This way, we may view each edge as a pair of \emph{half-edges} that  are oriented in opposing directions. The place where two consecutive half-edges at the boundary of a face meet is called a \emph{corner}. A corner is \emph{incident} to a unique vertex, and corners correspond bijectively to half-edges. In particular, rooting maps at an oriented root edge is equivalent to rooting maps at a corner. The \emph{degree} of a face is the number of half-edges on its boundary.

A planar map is called a  \emph{triangulation}, if each of its faces  has degree~$3$. There are more general versions of this definition where an exception is made for the outer face, but here we will only consider the case where the outer face has degree $3$ as well. A multi-graph is called \emph{simple}, if it has no multi-edges or loops. That is, between any unordered pair of distinct vertices there is at most one edge, and no edge starts and ends at the same vertex. We adapt the convention from~\cite{zbMATH06812193} and call a planar map simple, if its underlying multi-graph is simple. Hence a triangulation with at least $4$ vertices is simple if and only if it is $3$-connected.  The reader should take care, however, that different convention have been used in the literature. For example, Tutte~\cite{zbMATH03169204} additionally requires a simple triangulation to have no separating $3$-cycles. 

The asymptotic growth of the number $q_n$ of simple triangulations with $n + 2$ vertices (and hence $2n$ faces and $3n$ edges) was determined by~\cite{TUTTE1973437, zbMATH03169204}. It is given by
\begin{align}
	\label{eq:enumsimple}
	q_n = \frac{\sqrt{6}}{32 \sqrt{\pi}} n^{-5/2} \left( \frac{27}{256} \right)^{-n} \left(1 + O\left(\frac{1}{n}\right) \right). 
\end{align}

Let $\mQ_n$ denote the random planar map that is uniformly selected among all simple triangulation of the $2$-sphere with $n+2$ vertices.\footnote{It might seem more natural to use the letter $T$ to denote triangulations. However, the author uses the letter $T$ exclusively to denote trees.} For the purpose of proving local convergence of cubic planar graphs, we will require  quenched local convergence:
\begin{lemma}
	\label{le:convsimple}
	Let $(\mQ_n, c_n)$ denote the uniform simple $(n+2)$-vertex triangulation $\mQ_n$ re-rooted at a uniformly selected corner $c_n$.  There is a random locally finite simple triangulation $\hat{\mQ}$ such that
	\[
	\mathfrak{L}((\mQ_n, c_n) \mid \mQ_n) \convdis \mathfrak{L}(\hat{\mQ}).
	\]
	The convergence preserves the planar structure of the maps.
\end{lemma}

Angel and Schramm~\cite[Thm. 1.7]{MR2013797} constructed this infinite map, called the (type III) Uniform Infinite Planar Triangulation, and proved annealed local convergence of $(\mQ_n, c_n)$ towards it. 	The convergence preserves the planar structure in the sense that for any $k \ge 1$, the $k$-neighbourhood of $c_n$ in $\mQ_n$ converges in distribution as random finite corner-rooted planar map. We may formalize this as convergence of random elements of the space $\mathfrak{M}$ of locally finite (but possibly infinite) corner-rooted planar maps equipped with the local topology. That is, the projective limit topology arising from the projections to $k$-neighbourhoods interpreted as finite corner-rooted planar maps.

The quenched convergence in Lemma~\ref{le:convsimple} refers to distributional convergence of random elements of the collection $\mathbb{M}_1(\mathfrak{M})$ of Borel probability measures on $\mathfrak{M}$. There are various ways of proving convergence of $\mathfrak{L}((\mQ_n, c_n) \mid \mQ_n)$. One way would be to build on \cite[Thm. 1.7]{MR2013797} and use the fact that the limiting map is ergodic.\footnote{The author warmly thanks Justin Salez for this comment.} Alternatively, a more combinatorial approach would be to strengthen \cite[Thm. 1.7]{MR2013797}  using quenched results for counting submaps~\cite{MR2095934}. A third option is a direct approach that proves Lemma~\ref{le:convsimple} and recovers the annealed convergence. Specifically, as described in~\cite{zbMATH06812193}, the random simple triangulation $\mQ_n$ may be generated from a simply generated tree by adding edges and two vertices according to a closure operation by~\cite{zbMATH05126694}. By the continuous mapping theorem~\cite[Thm 2.7]{MR1700749}, quenched local convergence of the simply generated tree by~\cite{MR2908619} yields quenched local convergence of $\mQ_n$. These arguments are fully analogous to the proof in~\cite{quenchedmap} for quenched local convergence of large Boltzmann planar maps, which include the case of uniform unrestricted triangulations. Adapting the arguments of~\cite{quenchedmap} in the way we described to treat simple triangulations is straight-forward and yields no new insights, hence we leave the details to the inclined reader.

\subsection{Cubic $3$-connected planar maps and graphs}
\label{sec:cupam}

The \emph{dual map} of a planar map $M$ is the  ``red'' map obtained by placing a red vertex inside of each face of $M$ and then adding for each edge $e$ of $M$ a red edge between the red vertices corresponding to the two faces adjacent to $e$. This creates loops, if the two faces are identical. If $e$ is the root edge of $M$, we orient the corresponding red edge in a canonical way. This way, the dual map has an oriented root edge as well.

Recall that a triangulation with at least $4$ vertices is simple if and only if it is $3$-connected. It is well-known that the dual map construction yields a bijection between $3$-connected triangulations and  $3$-connected cubic planar maps. By Whitney's theorem, any such map has precisely two embeddings in the plane. Let $\cM(x,y)$ denote the exponential generating function of labelled $3$-connected cubic planar graphs that are rooted at a directed edge, with $x$ marking vertices and $y$ marking edges. Using the notation from Equation~\eqref{eq:enumsimple}, it follows that
\begin{align}
	\label{eq:M}
	\cM(x,y) = \frac{1}{2} \sum_{n \ge 2} q_n x^{2n} y^{3n}.
\end{align}

The type III UIPT $\hat{\mQ}$ almost surely belongs to the subset $\mathfrak{M}_0 \subset \mathfrak{M}$  of locally finite simple triangulations. We may consider the function $\psi: \mathfrak{M} \to \mathfrak{M}$ that sends a planar map $M$ to its dual map, if the dual map lies in $\mathfrak{M}$, and to some fixed place-holder value if it doesn't. This way, $\psi$ is continuous on the subset $\mathfrak{M}_0$. It follows by the  continuous mapping theorem~\cite[Thm 2.7]{MR1700749} that for any sequence of Borel probability measures $P_1, P_2, \ldots \in \mathbb{M}_1(\mathfrak{M})$ satisfying the weak convergence $P_n \Rightarrow \mathfrak{L}(\hat{\mQ})$, the push-forward measures satisfy
\begin{align*}
	P_n \psi^{-1} \Rightarrow \mathfrak{L}(\hat{\mQ}) \psi^{-1},
\end{align*}
and $\mathfrak{L}(\hat{\mQ}) \psi^{-1} = \mathfrak{L}(\psi(\hat{\mQ}))$. Here $\Rightarrow$ denotes weak convergence of probability measures.
In other words, the function
\[
	\mathbb{M}_1(\mathfrak{M}) \to \mathbb{M}_1(\mathfrak{M}), \quad P \mapsto P \psi^{-1},
\]
that maps a measure to its push-forward along $\psi$, is continuous at the point $\mathfrak{L}(\hat{\mQ})$. Applying the continuous mapping theorem~\cite[Thm 2.7]{MR1700749}  to the distributional convergence of random probability measures in Lemma~\ref{le:convsimple}, it follows that
\begin{align*}
	\mathfrak{L}((\mQ_n, c_n) \mid \mQ_n) \psi^{-1} \convdis \mathfrak{L}(\psi(\hat{\mQ})).
\end{align*}
The push-forward of $\mathfrak{L}((\mQ_n, c_n) \mid \mQ_n) \psi^{-1}$ is the uniform measure on the $6n$-element collection of corner-rooted versions of $\psi(\mQ_n)$. We set
\begin{align}
	\hat{\mM} = \psi(\hat{\mQ}).
\end{align}
We have thus verified:

\begin{lemma}
	\label{le:conv3cubic}
	Let $\mM_n$ denote the uniformly selected $3$-connected cubic planar graph with $2n$ vertices and $3n$ edges. Let $c_n$ denote a uniformly selected corner of $\mM_n$. Then
	\begin{align*}
			\mathfrak{L}((\mM_n, c_n) \mid \mM_n) \convdis \mathfrak{L}(\hat{\mM}).
	\end{align*}
\end{lemma}

\section{Quenched local convergence of connected cubic planar graphs}
\label{sec:cuplag}

\subsection{Network decomposition}
\label{sec:netdec}

A \emph{(cubic) network} is a connected planar cubic multi-graph $N$ with an oriented root edge $e$ such that the graph $N - e$ obtained by removing the root edge  is simple. The vertices of $N$, including the endpoints of $e$, are labelled. We refer to the endpoints as the \emph{poles} of the network. The exponential generating function $\cN(x)$ of the class of networks is defined so that for all $n \ge 0$ the coefficient $[x^n]\cN(x)$  equals $1/n!$ multiplied the number of networks with $n$ vertices. Of course, this coefficient equals zero unless $n \in \{2i \mid i \ge 2\}$.

The class $\cN$ of networks may be partitioned into five disjoint subclasses:
\begin{enumerate}
	\item  $\cL$ (Loop). The root edge is a loop.
	\item $\mathcal{I}$ (Isthmus). The root edge is an isthmus, meaning $N - e$ is disconnected.
	\item $\mathcal{S}$ (Series). $N - e$ is connected, but contains a bridge that separates the endpoints of $e$.
	\item $\mathcal{P}$ (Parallel). $N - e$ is connected, contains no bridge that would separate the endpoints of $e$, and either $e$ is part of a double edge in $N$ or deleting the endpoints of $e$ disconnects $N$.
	\item $\mathcal{H}$ (Polyhedral). $N$ is obtained from a $3$-connected network by possibly replacing each non-root edge with a non-isthmus network.
\end{enumerate}

It is a non-trivial fact that these are the only classes that need to be considered. We refer the reader to~\cite{zbMATH05122852, zbMATH07213288} for a detailed justification.  In the following we recall the decomposition of the individual classes following closely the presentation in these references. We will often drop the argument of the generating series, writing $\cN$ instead of $\cN(x)$. It will be notationally convenient to introduce the subclass $\cD$ such that
\begin{align}
	\label{eq:arga}
	\cN &= \cD + \cI, \\
	\label{eq:ddd}
	\cD &= \cL + \cS + \cP + \cH.
\end{align}

\subsubsection{Loop networks}

A cubic network $N$ belongs to the class $\cL$ of loop-networks if its root edge $st$ is a loop. That is, if $s=t$. As illustrated in Figure~\ref{fi:loop}, the vertex $s$ of the loop is adjacent to a single vertex $s'$, which is adjacent to two distinct vertices $u \ne v$ that form the poles of a non-loop network. There are two ways to orient this associated network, yielding
\begin{align}
	\label{eq:argb}
	\cL= \frac{x^2}{2}(\cN - \cL).
\end{align}

\begin{figure}[H]
	\centering
	\begin{minipage}{\textwidth}
		\centering
		\includegraphics[scale=0.7]{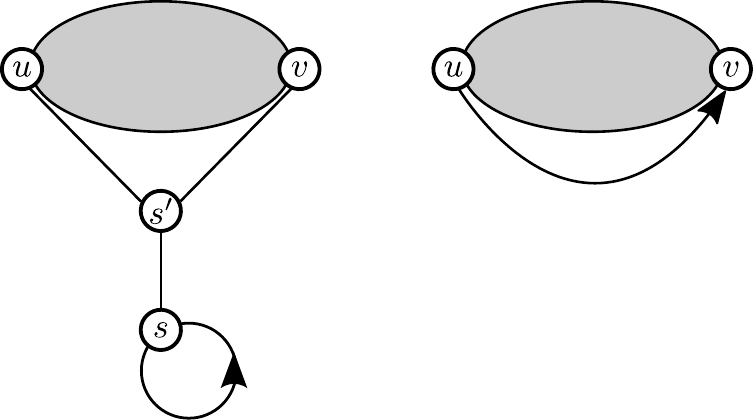}
		\caption{Decomposition of loop networks.}
		\label{fi:loop}
	\end{minipage}
\end{figure}

\subsubsection{Isthmus networks}

As illustrated in Figure~\ref{fi:isthmus}, an isthmus network corresponds to an ordered pair of loop networks, each having an additional vertex. Thus,
\begin{align}
	\label{eq:argc}
	\cI = \frac{\cL^2}{x^2}.
\end{align}

\begin{figure}[H]
	\centering
	\begin{minipage}{\textwidth}
		\centering
		\includegraphics[scale=0.7]{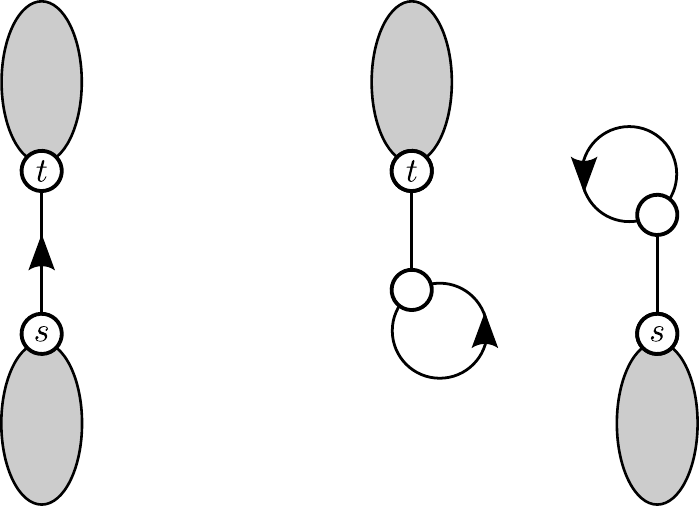}
		\caption{Decomposition of isthmus networks.}
		\label{fi:isthmus}
	\end{minipage}
\end{figure}

\subsubsection{Series networks}

If $N$ is a series network, then $N - e$ contains one or more bridges that separate the poles $s$ and $t$. Let $uv$ denote the bridge that is closest to $s$, directed from $u$ to $v$ such that $u$ is closer to $s$ than $v$. As illustrated in Figure~\ref{fi:series}, $N$ corresponds to two networks $N_1$ and $N_2$, with root edges $su$ and $vt$. If $s=u$, then $N_1$ is a loop network, and likewise if $v=t$ then $N_2$ is a loop network. Both $N_1$ and $N_2$ cannot be isthmus networks, since there are multiple paths between their poles. Since we chose the bridge $uv$ which is closest to $s$, the network $N_1$ additionally cannot be a series network. Hence,
\begin{align}
	\label{eq:S}
	\cS = (\cN- \cS - \cI)(\cN - \cI).
\end{align}
Note that the subclass of series networks whose root edge is a double edge is given by $\cL^2$.

\begin{figure}[H]
	\centering
	\begin{minipage}{\textwidth}
		\centering
		\includegraphics[scale=0.7]{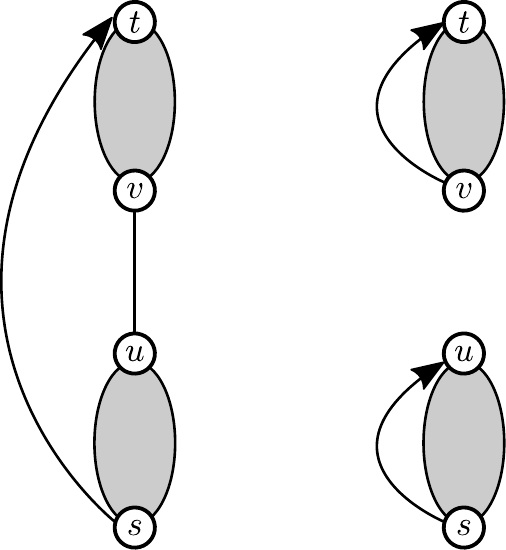}
		\caption{Decomposition of series networks.}
		\label{fi:series}
	\end{minipage}
\end{figure}

\subsubsection{Parallel networks}

There are two types of parallel networks. If the root edge is a double edge, then the poles $s$ and $t$ are adjacent to (possibly identical) vertices $u$ and $v$, as illustrated in the first row of Figure~\ref{fi:parallel}. The corresponding smaller network with poles $u$ and $v$ cannot be an isthmus network. If the root edge is not a double edge, then as illustrated in the second row of Figure~\ref{fi:parallel} the parallel network corresponds to an unordered pair of two non-isthmus networks. Hence
\begin{align}
	\label{eq:decparallel}
	\cP= x^2 \cD + \frac{x^2}{2}\cD^2.
\end{align}
The summand $x^2 \cD$ also corresponds precisely to the parallel networks whose root edge is a double edge.

\begin{figure}[H]
	\centering
	\begin{minipage}{\textwidth}
		\centering
		\includegraphics[scale=0.7]{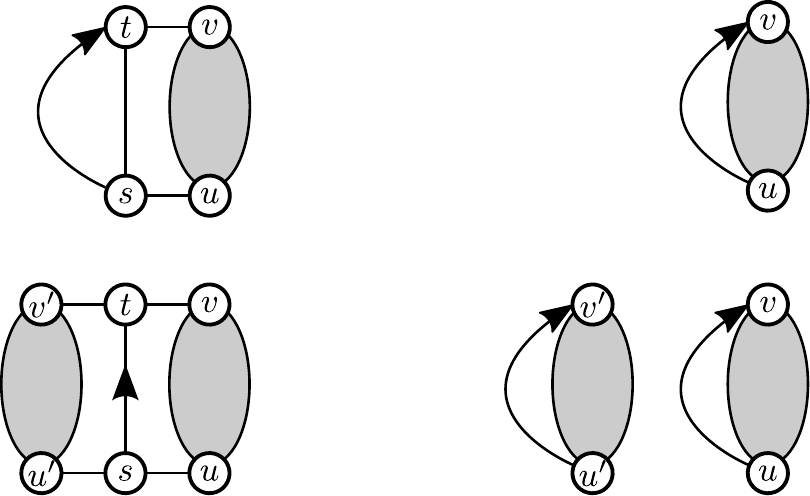}
		\caption{Decomposition of the two types of parallel networks.}
		\label{fi:parallel}
	\end{minipage}
\end{figure}

\subsubsection{Polyhedral networks}

A polyhedron network $N$ consists of a $3$-connected cubic planar graph $M$ with a directed root edge, together with components $D(1), D(2), \ldots$, one for each non-root edge of $M$. Here we choose, for each possible shape of $M$, a canonical enumeration and orientation for each non-root edge.  Each component is either an edge or a $\cD$-network. The network $N$ is obtained from the core by inserting at each canonically directed non-root edge $uv$ of the core the corresponding component $D(i)$. Here \emph{inserting} means doing nothing if the component is an edge. If the component $D(i)$ is a network, we delete the edge $uv$ from $M$, delete the root edge of $D(i)$, and insert an edge between $u$ and the south pole of $D(i)$, and another edge between $v$ and the north pole of $D(i)$. We say $M$ is the $3$-connected core of $N$.  See Figure~\ref{fi:poly} for an illustration. This entails
\begin{align}
	\label{eq:decpoly}
	\cH= \frac{\cM(x, 1 + \cD)}{1 + \cD}.
\end{align}

\begin{figure}[H]
	\centering
	\begin{minipage}{\textwidth}
		\centering
		\includegraphics[scale=0.7]{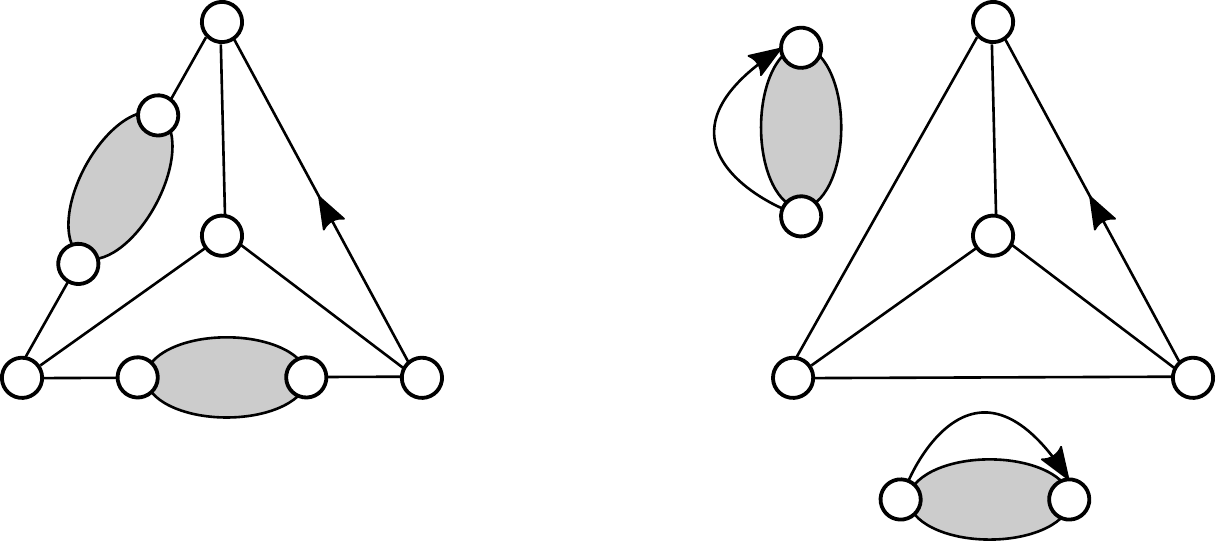}
		\caption{Decomposition of polyhedra networks.}
		\label{fi:poly}
	\end{minipage}
\end{figure}

\subsubsection{Simple networks}
	\label{simplenetworks}
	A network is simple, if it is not a loop network,  not the kind of parallel network illustrated in the top half of Figure~\ref{fi:parallel}, and not a series network whose two components are loop networks. Thus, the class $\cN_{\mathrm{s}}$ of simple networks may be decomposed as follows:
\begin{align}
	\label{eq:simpnetwork}
	\cN_{\mathrm{s}} 	&= \cH +  \frac{x^2}{2} \cD^2 + \cI + (\cS - \cL^2).
\end{align}
Simple networks are precisely connected cubic planar graphs with an oriented root edge.

\subsection{The largest $3$-connected component}
\label{sec:largest}

The present section provides a proof of Theorem~\ref{te:main2}. We are going to require singular expansions for the generating series introduced in the previous section. We recall these expansions following closely the presentation in~\cite{zbMATH07213288}.

We let $\cQ(z) = \sum_{n \ge 1} q_n z^n$ denote the generating series for simple triangulations, with $q_n$  denoting the number of simple triangulations with $n+2$ vertices, as in  Equation~\eqref{eq:enumsimple}. Tutte showed in~\cite{zbMATH03169204} that \[
\cQ(z) = U(z)(1 - 2U(z)),\] where $U(z)$ denotes a power series with non-negative coefficients satisfying the equation \[U(z)(1- U(z))^3 = z.\] Solving this equation  and setting $\tau= 27/256$ and $Z =\sqrt{1 - z/\tau}$, it follows that
\begin{align}
	\label{eq:expq}
	\cQ(z) =  \frac{1}{8}-\frac{3 Z^2}{16}+\frac{Z^3}{4 \sqrt{6}}-\frac{13
		Z^4}{192}+\frac{35 Z^5}{288 \sqrt{6}}-\frac{1201 Z^6}{31104}+ O(Z^7)
\end{align}
as $z \to \tau$. By Equation~\eqref{eq:M}, 
\[
\cM(x,y) = \frac{1}{2}( \cQ(x^2y^3) - x^2y^3 ).
\]
We also define the generating series $\bar{\cM}(x,y)$ of $3$-connected cubic planar graphs without a directed root edge, which satisfies
\[
\cM(x,y) = 2 y\frac{\partial \bar{\cM}}{\partial y} (x,y).
\]
Combining Equations~\eqref{eq:arga},~\eqref{eq:argb}, and~\eqref{eq:argc} yields
\[
	\cL = \frac{x^2}{2}(\cD + \cI - \cL) = \frac{x^2}{2}(\cD + \cL^2 / x^2 - \cL).
\]
This equation may be solved for $\cL(x)$, yielding
\begin{align}
	\label{eq:plug1}
\cL= 1 + \frac{x^2}{2} - \sqrt{\frac{x^2}{4} + 1 -x^2(\cD -1)}.
\end{align}
Equation~\eqref{eq:S} may be rewritten as
$\cS= (\cD - \cS)\cD$.  Solving for $\cS$ yields
\begin{align}
	\label{eq:plug2}
	\cS = \frac{\cD^2}{1+ \cD}.
\end{align}
Plugging~\eqref{eq:decparallel},~\eqref{eq:decpoly},~\eqref{eq:plug1} and~\eqref{eq:plug2} into Equation~\eqref{eq:ddd} yields
\begin{align}
	\label{eq:ford}
	\cD = 1 + \frac{x^2}{2} - \sqrt{\frac{x^2}{4} + 1 -x^2(\cD -1)} + \frac{\cD^2}{1+ \cD} + x^2 \cD + \frac{x^2}{2}\cD^2 +  \frac{\cM(x, 1 + \cD)}{1 + \cD}.
\end{align}
As argued in~\cite[Proof of Thm.1]{zbMATH07213288}, this recursive equation allows us to derive a singular expansion for  the exponential generating series $\cD(x)$. It has a radius of convergence $\rho>0$ that satisfies
\begin{align}
	\label{eq:rhotau}
	\rho^2(1 + \cD(\rho))^3 = \tau.
\end{align}
The two points $\{\rho, - \rho\}$ are the only singularities on the circle $\{w \in \ndC \mid |w|=\rho\}$. The series $\cD(x)$ satisfies the singular expansion
\begin{align}
	\label{eq:expd}
	\cD(x) = D_0 + D_2 (1 - x/\rho) + D_{3} (1 - x/\rho)^{3/2} + O( (1-x/\rho)^2 )
\end{align}
as $x \to \rho$, for constants $D_0 = \cD(\rho)$, $D_2 = - \rho \cD'(\rho)$, $D_3>0$. In particular, by the arguments of~\cite[Proof of Thm. 1]{zbMATH07213288}, the following numerical approximations hold
\begin{align*}
	\rho &=       0.319224606195452700761429068280\ldots, \\
	\cD(\rho) &=  0.011525944379127380775581944095\ldots, \\
	\cD'(\rho) &= 0.370296056465161996287563244273\ldots, \\
	D_3 &=  	  0.254267214080405673433969610493\ldots.
\end{align*}
The singularity expansion~\eqref{eq:expd} entails
\begin{align}
	\label{eq:dasymp}
	[x^n] \cD(x) \sim c_{\cD} n^{-5/2} \rho^{-n}, \qquad c_{\cD} = \frac{3 D_3}{2 \sqrt{\pi}}
\end{align}
as $n \in 2 \ndN$ tends to infinity. The series $\cC^\bullet(x)$ counts vertex-rooted cubic planar graphs. Any such graph has $3$ canonical choices for an oriented root edge that yields a simple network. Hence $\cC^\bullet(x)$ is related to the generating series $\cN_{\mathrm{s}}$ of simple networks from Equation~\eqref{eq:simpnetwork} via
\begin{align*}
	3 \cC^\bullet = \cN_s.
\end{align*}
Hence
\begin{align}
	\label{eq:cbulletasdf}
	3 \cC^\bullet &= \cD  + \cI - (x^2 \cD + \cL + \cL^2) \\
				&= (1 - x^2)\cD - \cL + (1/x^2 - 1) \cL^2. \nonumber
\end{align}
Plugging in~\eqref{eq:plug1} and~\eqref{eq:expd} yields 
\begin{align}
	\cC^\bullet = C_0^\bullet + C_2^\bullet(1 - x/\rho) + C_3^\bullet(1 - x/\rho)^{3/2} + O( (1 - x/\rho)^2 )
\end{align}
as $x \to \rho$, and hence
\begin{align}
	\label{eq:cbulletexp}
	n [x^n] \cC(x) = [x^n] \cC^\bullet(x) \sim \frac{3C^\bullet_3}{2\sqrt{\pi}} n^{-5/2} \rho^{-n}.
\end{align}

\begin{proof}[Proof of Thm.~\ref{te:main2}]
	
	Given a finite connected cubic planar graph $C$, we may select and orient a root edge in order to obtain a simple network. As described in Equation~\eqref{eq:simpnetwork}, this network is either a polyhedral network, a series network, an isthmus network, or a simple parallel network (that is, the summand $\frac{x^2}{2} \cD^2$ in Equation~\eqref{eq:simpnetwork}). Each of these networks again may be decomposed into smaller networks and atoms (which correspond to the vertices of $C$), as described in Section~\ref{sec:netdec}. Proceeding recursively, this decomposition algorithm terminates after a finite number of steps.  In this procedure, every time we encounter an $\cH$-network,  we decompose it into a $3$-connected cubic network with $(1 +\cD)$-components inserted at its non-root edges. The underlying $3$-connected cubic graphs corresponding to the $3$-connected cubic networks encountered in this way are the $3$-connected components of the connected cubic planar graph $C$. Their vertex sets are subsets of the vertex set of $C$. They do not depend on the choice of the oriented root edge $e$, only on $C$. Moreover, $C$ has at least one such component, since $\cH$-networks are the only networks that may be decomposed entirely into atoms.
	
	It is crucial to note that, as a direct consequence of the decomposition steps in Section~\ref{sec:netdec},  the $3$-connected components of $C$ do not overlap. That is, no two $3$-connected components share a common vertex. This  allows us to use a double rooting argument similar to that of~\cite{zbMATH01308674}:

	With foresight, we define the constant
	\begin{align}
		\label{eq:kappa}
		\kappa &= \frac{2(1 + \cD(\rho))}{2(1 + \cD(\rho)) + 3 \cD'(\rho)\rho} \\
		&= 0.850853090058314333870385348879612617197477\ldots. \nonumber
	\end{align} 
	Let $M>0$ be a constant and let $k  = \kappa n + t n^{2/3} \in 2 \ndN$ with $t \in [-M, M]$. Let $c_{n,k}$ denote the number of cubic planar graphs with $n$ vertices and a marked vertex that belongs to a $3$-connected component with $k$ vertices. Note that not every vertex needs to belong to  a $3$-connected component, but there are always some that do. Let $c_{n,k}^*$ denote the number of cubic planar graphs with an arbitrary marked vertex whose largest component has $k$ vertices. Note that, since $k > n/2$, in both cases such a component is unique.
	Hence we have
	\[
		n c_{n,k} = k c_{n,k}^*,
	\]
	because both sides of the equation represent the number of $n$-vertex cubic planar graphs with a marked vertex and a second marked vertex that belongs a $k$-sized $3$-connected component.
	
	Thus, with $c_n$ denoting the number of $n$-vertex labelled cubic planar graph, the probability for the largest $3$-connected component in $\mC_n$ to have $k$ vertices is given by
	\begin{align}
		\label{eq:tosimplify}
		\Pr{V_n = k} = \frac{c_{n,k}^*}{n c_n} = \frac{n}{k} \frac{c_{n,k}}{n c_n}.
	\end{align}
	By choice of $k$, we have
	\begin{align}
		\frac{n}{k} \sim \frac{1}{\kappa}.
	\end{align}
	By~\eqref{eq:cbulletexp} we have
	\[
		\frac{n c_n}{n!} \sim \frac{3C^\bullet_3}{2\sqrt{\pi}} n^{-5/2} \rho^{-n}.
	\]
	Plugging \eqref{eq:plug1} into \eqref{eq:cbulletasdf} and using the singular expansion~\eqref{eq:expd} for $\cD(x)$ yields
	\begin{align}
		\label{eq:cbullet3}
		C^\bullet_3 = \frac{-D_3 \rho^4-2 D_3 \rho^2+2 D_3}{3
			\sqrt{-4 D_0 \rho^2+\rho^4+4 \rho^2+4}}.
	\end{align}
	Furthermore,
	\begin{align*}
		\frac{c_{n,k}}{n!} &= k [x^n u^{3k/2}] \bar{\cM}( x, u(1 + \cD(x)) ).
	\end{align*}
	A $3$-connected cubic planar graph with $3k/2$ edges has $3k$ versions with an oriented root edge. Hence
	\begin{align*}
		k [x^n u^{3k/2}] \bar{\cM}( x, u(1 + \cD(x)) ) &= \frac{1}{3} [x^n u^{3k/2}] \cM( x, u(1 + \cD(x)) ) \\&= \frac{1}{6}[x^n u^{3k/2}] \cQ^+\left(x^2u^3(1 + \cD(x))^3\right) \\
		&= \frac{1}{6}[x^{n/2} u^{k/2}] \cQ^+\left(xu(1 + \cD(\sqrt{x}))^3\right),
	\end{align*}
with $\cQ^+(z) := \cQ(z) - z$. The singular expansions~\eqref{eq:expq} and~\eqref{eq:expd} entail
\begin{align*}
	\cQ^+ = Q^+_0 + Q^+_2(1 - x/\tau) + Q^+_3(1 - x/\tau)^{3/2} + O( (1 - x/\tau)^2 )
\end{align*}
as $z \to \tau$, with
\begin{align*}
	Q^+_0 =  \frac{5}{256}, \quad Q^+_2 = \frac{21}{256}, \quad Q^+_3 =  \frac{1}{4 \sqrt{6}},
\end{align*}
and, using~\eqref{eq:rhotau}, 
\begin{align*}
	z(1 +\cD(\sqrt{z}))^3 = A_0 + A_2 (1 - z / \rho^2) + A_3 (1 - z / \rho^2)^3
\end{align*}
as $z \to \rho^2$, with
\begin{align*}
	A_0 = \tau, \quad A_2 =  \left(\frac{3D_2}{2(D_0+1)}   - 1  \right) \tau, \quad A_3 =  \frac{3  D_3 \tau }{2 \sqrt{2} (D_0 + 1)} .
\end{align*}
This allows us to apply a general result for asymptotics of  composition schemes~\cite[Thm. 5, (ii)]{MR1871555}.  Applying this result and using $k/2 = \kappa n/2 + t/2$ it follows that
\begin{align*}
	[x^{n/2} u^{k/2}] \cQ^+\left(xu(1 + \cD(\sqrt{x}))^3\right) \sim \frac{3}{4 \sqrt{\pi n^5 / 2^5}} \rho^{-n}  \alpha_0^{-3/2} Q_3^+ c2^{1/3} h(c t/2^{1/3})
\end{align*}
with
\begin{align*}
	\alpha_0 = \frac{\tau}{-A_2} = \kappa
\end{align*}
and 
\begin{align*}
	c = \frac{1}{\kappa} \left( \frac{-A_2}{3A_3} \right)^{2/3} =  \frac{D_3 \left(\frac{-2 D_0+3
			D_2-2}{D_3}\right)^{5/3}}{3\ 2^{2/3} \sqrt[3]{3}
		(D_0+1)}.
\end{align*}
Thus, Equation~\eqref{eq:tosimplify} simplifies to 
\begin{align}
	\Pr{V_n = k} \sim C c 2^{-1/3} h(ct 2^{-1/3}) n^{-2/3}
\end{align}
with
\begin{align}
	\label{eq:evalthec}
	C &= \frac{1}{6}  \frac{Q_3^+}{C_3^\bullet} \kappa^{-5/2} 2^{5/2}  \\
	&= \frac{\sqrt{-4 D_0 \rho^2+\rho^4+4
			\rho^2+4}}{8 \sqrt{6} \left(\frac{D_0+1}{2 D_0-3
			D_2+2}\right)^{5/2} \left(-D_3 \rho^4-2 D_3
		\rho^2+2 D_3\right)}. \nonumber
\end{align}
This constant may be  evaluated algebraically to equal
\[
	C = 2.
\]
Indeed, Equation~\eqref{eq:ford} implies
\begin{align}
	\label{eq:tosolve}
	(1+\cD(x))\sqrt{x^4/4 + 1 - x^2(\cD(x)-1)} - 1 - (1/2)\cQ(x^2(1+\cD(x))^3) = 0.
\end{align}
Plugging $\cQ(\rho^2(1 + \cD(\rho))^3) = 1/8$ into~\eqref{eq:tosolve} and using Mathematica to solve the result simultaneously with $\rho^2(1 + \cD(\rho))^3 = \tau$ for $(\rho, \cD(\rho))$ yields algebraic expressions for these two constants. Differentiating~\eqref{eq:tosolve} and solving for $\cD'(\rho)$ yields an algebraic expression of $\cD'(\rho)$ in terms of $\rho$ and $\cD(\rho)$. Furthermore, we may combine~\eqref{eq:expq} with~\eqref{eq:expd} to calculate the Puiseux expansion of the left hand side of ~\eqref{eq:tosolve}. The coefficient of $(1 - x/\rho)^3$ needs to equal zero, allowing us to solve for $D_3$ to obtain an algebraic expression of $D_3$ in terms of $\rho$ and $\cD(\rho)$. Plugging these algebraic expressions into~\eqref{eq:evalthec} and simplifying the result with Mathematica's \texttt{FullSimplify} function yields $C=2$.

This completes the proof of Theorem~\ref{te:main2} with $\kappa$ given in Equation~\eqref{eq:kappa} and
\begin{align}
	\label{eq:cvconst}
		c_v &= \frac{D_3 \left(\frac{2 D_0-3
				D_2+2}{D_3}\right)^{5/3}}{6 \sqrt[3]{3}
			(D_0+1)} \\
		&= 1.205660773457703954344217302817493214574105705427\ldots \nonumber.
\end{align}
\end{proof}

\subsection{The components attached to the $3$-connected core}
\label{sec:components}

We let $Y \ge 0$ denote a random non-negative integer with probability generating function \begin{align}
	\label{eq:noway}
	\Ex{z^{Y}} = \frac{1 + \cD(\rho z)}{1 + \cD(\rho)}.
\end{align}
Note that
\begin{align}
	\label{eq:exy}
	\Ex{Y} = 	\frac{2}{3}\left(\frac{1}{\kappa} -1\right).
\end{align}

We also  define the corresponding network, which will play a major role in determining the asymptotic shape of the random connected cubic planar graph $\mC_n$.
\begin{definition}[Boltzmann  network]
	\label{de:defd}
	Let $\mD$ denote a random $(1+\cD)$-structure that is uniformly selected among all $Y$-sized $(1+\cD)$-structures.
	
	That is, if $Y>0$, then $\mD$ is a random non-isthmus network with $Y$ vertices. If $Y=0$, $\mD$ is equal to  a place-holder value representing that inserting $\mD$ at an edge of another network  has no effect. In other words, it leaves the network unchanged.
\end{definition}

By Equation~\eqref{eq:dasymp} it follows that
\begin{align}
	\label{eq:dtail}
	\Pr{Y=n} \sim \frac{c_{\cD}}{1 + \cD(\rho)} n^{-5/2}
\end{align}
as $n \in 2 \ndN$ tends to infinity.

Let $(Y_i)_{i \ge 1}$ denote independent copies of $Y$. Let $X_n = 3V_n /2$ denote the number of edges of the $3$-connected core of $\mC_n$. We let $ (Z_i)_{1 \le i \le X_n}$ denote the numbers of vertices of the components $(\cD_i(\mC_n))_{1 \le i \le X_n}$ of $\mC_n$. Let $k \ge 1$ be an integer. Conditional on having a $3$-connected core with  $X_n = 3k$ edges, the component sizes follow the distribution of the conditioned vector
\begin{align}
	\label{eq:zy}
	\left( (Z_i)_{1 \le i \le 3k} \mid X_n = 3k \right) \eqdist \left( (Y_i)_{1 \le i \le 3k} \Big\vert \sum_{i=1}^{3k} Y_i = n-2k \right).
\end{align}

With this notation at hand, we are ready to  establish a contiguity relation between $\mC_n$ and a model where all but a small number of components attached to the $3$-connected core of $\mC_n$ are replaced by independent copies of the Boltzmann random network $\mD$.

\begin{proof}[Proof of Theorem~\ref{te:main3}]
	By Equation~\eqref{eq:dtail} the random variable $Y/2$ follows asymptotically a power law with index $k^{-5/2}$. Hence the classical local limit theorem~\cite[Thm. 4.2.1]{MR0322926} entails that there exists a constant $c_1>0$ such that
	\begin{align}
		\label{eq:llty}
		\Prb{\frac{1}{2} \sum_{i=1}^k  Y_i = \frac{k}{2}\Exb{Y} + t k^{2/3}} = \frac{1}{k^{2/3}} \left( o(1) + c_1 h(c_1 t) \right)
	\end{align}
	as $k \to \infty$, uniformly for all $t$ such that $\frac{k}{2}\Exb{Y} + t k^{2/3} \in \ndN$. The density function $h$ is positive, uniformly continuous and bounded on $\ndR$.

	For all $M_1, M_2>0$ we define the collection $\cE_{n, \delta, M_1,M_2}$ of finite sequences \[
	E = (3k, y_1, \ldots, y_{3k - \lfloor \delta n \rfloor })
	\]  with  integers $k, y_1, y_2, \ldots \ge 0$ satisfying the following properties:
	\begin{enumerate}[\qquad a)]
		\item $\Pr{X = 3k}>0$  and $\Pr{Y=y_i} >0$ for all  $1 \le i \le 3k - \lfloor \delta n \rfloor $,
		\item $|\sum_{i=1}^{3k - \lfloor \delta n \rfloor} y_i - (3k - \delta n) \Ex{Y} | \le  M_1 n^{2/3}$,
		\item $|k(3 \Ex{Y} + 2) - n| \le   M_2  n^{2/3}$.
	\end{enumerate}
	For any such sequence $E$ we set $\ell = \sum_{i=1}^{3k - \lfloor \delta n \rfloor } y_i$.  By Equation~\eqref{eq:zy}, it holds that
	\begin{align}
		\label{eq:quotient}
		\frac{\Prb{ (X_n, (Z_i)_{1\le i \le  X_n - \lfloor \delta n \rfloor })  = E } } {\Prb{ (X_n, (Y_i)_{1\le i \le  X_n - \lfloor \delta n \rfloor })  = E  }}		
		&= \frac{\Prb{ \ell + \sum_{i=3k-\lfloor \delta n\rfloor +1}^{3k} Y_i = n-2k }}{ \Prb{\sum_{i=1}^{3k}  Y_i = n - 2k}  } \\
		&=  \frac{\Prb{ \sum_{i=1}^{\lfloor \delta n\rfloor } Y_i = n-2k-\ell }}{ \Prb{\sum_{i=1}^{3k}  Y_i = n - 2k}  }. \nonumber
	\end{align}
Here we have  implicitly used Assumption~a), which  ensures that we don't divide by zero.
	By Equation~\eqref{eq:llty}, it follows that
	\begin{multline*}
		\frac{\Prb{ \sum_{i=1}^{\lfloor \delta n\rfloor } Y_i = n-2k-\ell }}{ \Prb{\sum_{i=1}^{3k}  Y_i = n - 2k} } = \\  \left( \frac{3k}{\lfloor \delta n \rfloor } \right)^{2/3} \frac{o(1) + c_1 h \left( \frac{c_1}{2(\lfloor \delta n\rfloor)^{2/3} }(n - 2k - \ell - \lfloor \delta n \rfloor \Ex{Y} )   \right)}{o(1) + c_1 h \left(\frac{c_1}{2(3k)^{2/3}}(n - 2k - 3k\Ex{Y})  \right)}
	\end{multline*}
	with uniform $o(1)$ terms.
	Assumptions~b) and c) entail that the arguments of the density $h$ in both the numerator and denominator lie in a compact interval whose upper and lower bound only depend on $M_1$, $M_2$, and $\delta$. Since the density $h$ is positive,  continuous, and bounded on $\ndR$, it follows that the quotient also belongs to a compact interval whose bounds only depend on $M_1$, $M_2$, and $\delta$. We emphasize that the upper bound, let us denote it by $U(\delta, M_2)$, 
	in fact only depends on $\delta$ and $M_2$, but not on $M_1$. This verifies the existence of constants $0<c<C$ such that Inequality~\eqref{eq:N2} holds uniformly for all $E \in \cE_n := \cE_{n, \delta, M_1,M_2}$.
	
	The limit theorems~\eqref{eq:vclt},~\eqref{eq:llty}, and the expression~\eqref{eq:exy} for $\Ex{Y}$ 
	  readily imply that, given $\epsilon_1 >0$,  we may choose $M_1> 0$ large enough so that
	\[
	\Prb{ \left|\sum_{i=1}^{3X_n - \lfloor \delta n \rfloor} Y_i - (3X_n - \delta n) \Ex{Y} \right| >  M_1 n^{2/3}} < \epsilon_1,
	\]
	and given $\epsilon_2>0$ we may choose $M_2>0$ large enough so that
	\[
	\Prb{|X_n(3 \Ex{Y} + 2) - n| >  M_2  n^{2/3}} < \epsilon_2.
	\]
	Furthermore, it follows by the bounds we have shown so far that
	\[
	\Prb{ \left|\sum_{i=1}^{3X_n - \lfloor \delta n \rfloor} Z_i - (3X_n - \delta n) \Ex{Y} \right| >  M_1 n^{2/3}} < U(\delta, M_2) \epsilon_1.
	\]
	Hence, choosing $M_2$ for $\epsilon_2 = \epsilon/2$, and then $M_1$ for $\epsilon_1 = \max(\epsilon / 2, \epsilon/ (2 U(\delta, M_2)))$, it follows that~\eqref{eq:doa} and~\eqref{eq:dob} hold. This completes the proof.	
\end{proof}

We let $\mD_n$ denote the uniform $\cD$-network with $n \in 2\ndN$ vertices. That is, $\mD_n$ is distributed like  $\mD$ conditioned on having $n \ge 1$ vertices.  We let $V_n^{\cD}$ denote the number of vertices of its largest $3$-connected component. This way, we may view $\mD_n$ as the result of blowing up the $3V_n^\cD/2$ edges of its $3$-connected core $\cM(\mD_n)$ by non-isthmus networks $(\cD_i(\mD_n))_{1 \le i \le 3 V_n/2}$. Here we may choose any canonical order of the components, and hence may assume that the root edge of $\mD_n$ always belongs to $\cD_{3 V_n^\cD/2}(\mD_n)$. Thus, $\cD_{3 V_n^\cD/2}(\mD_n)$ either is empty, meaning the root edge of $\mD_n$ coincides with the root edge of the $3$-connected core, or it has an inner and an outer oriented root edge, both of which may be double edges, none of which may be isthmuses. We let $\cD^*(x)$ denote the generating series of such networks.

\begin{proposition}
	\label{pro:doit}
	The statements of Theorem~\ref{te:main2} and Theorem~\ref{te:main3} also hold for $\mD_n$ instead of $\mC_n$.
\end{proposition}

The proof of Proposition~\ref{pro:doit} is almost identical to that of Theorem~\ref{te:main2} and Theorem~\ref{te:main3}, hence we omit the details. The only difference is that the component $\cD_{3 V_n^\cD/2}(\mD_n)$ needs a special treatment because it contains two roots. Hence, instead of the composition scheme $\bar{\cM}( x, u(1 + \cD(x)) )$, we have to use  the composition scheme
\begin{align}
	\frac{\cM( x, u(1 + \cD(x)) )}{1 + \cD(x)} (1 + \cD^*(x)).
\end{align}
Likewise, instead of Equation~\eqref{eq:zy}, the distribution of the numbers of vertices $(Z_i^\cD)_{1 \le i \le 3k}$ of the networks attached to the $3$-connected core of $\mD_n$ conditional on the event, that the number $X_n^\cD$ of edges in the core equals a given integer $3k$, is given by
\begin{align}
	\left( (Z_i)_{1 \le i \le 3k} \mid X_n^\cD = 3k \right) \eqdist \left( (Y_1, \ldots, Y_{3k-1}, Y^\cD) \Big\vert \sum_{i=1}^{3k-1} Y_i + Y^\cD = n-2k \right),
\end{align}
with $Y^{\cD}$ an independent random integer with probability generating function
\begin{align}
	\Ex{z^{Y^\cD}} = \frac{1 + \cD^*(\rho z)}{1 + \cD^*(\rho)}.
\end{align}
Apart from that, the arguments of Theorem~\ref{te:main2} and Theorem~\ref{te:main3} may be copied almost word by word.

We may now finalise the proof of  Corollary~\ref{co:second} by justifying the stochastic lower bound in~\eqref{co:second}.

\begin{proof}[Proof of Corollary~\ref{co:second}]
	We already explained in the introduction how Equation~\eqref{eq:uppera} follows from Theorem~\ref{te:main3}. As each non-maximal $3$-connected component must be part of one of the components $\cD_i(\mC_n)$, $1 \le i \le 3 V_n /2$, this  readily yields
	\[
		V_n^{(2)}  = O_p(n^{2/3}).
	\]
	 As for the lower bound, Proposition~\ref{pro:doit} implies that the number of vertices $V_n^\cD$ in the largest $3$-connected component of $\mD_n$ satisfies
	\begin{align}
		\label{eq:concentrationvnd}
		V_n^\cD / n \convp \kappa.
	\end{align}
	Let $(t_n)_{n \ge 1}$ denote an arbitrary sequence satisfying $t_n \to \infty$. By Equation~\eqref{eq:dtail} and standard extremal value statistics the largest component in $n$ independent copies of $\mD$ is larger than $n^{2/3}/t_n$ with probability tending to $1$ as $n$ tends to infinity. Hence, by~\eqref{eq:concentrationvnd}, the largest $3$-connected core in these $n$ samples is larger than $(n^{2/3} / t_n)\kappa /2$ with probability tending to $1$. By Theorem~\ref{te:main3} it follows that
	\[
	\Pr{V_n^{(2)} \ge ((\kappa n /2)^{2/3} / t_n)\kappa /2} \to 1
	\]
	as $n \in 2\ndN$ tends to infinity.
	Since $t_n$ was arbitrary, Equation~\eqref{eq:lowerb} follows, and the proof is complete.
\end{proof}

We conclude this section with some remarks on Theorem~\ref{te:main3}. As mentioned in the introduction, the bounds $c$ and $C$ in Inequality~\eqref{eq:N2}  become worse if $\delta \to 0$. Moreover,  a stronger approximation of $(X_n, (\cD_i(\mC_n))_{1\le i \le 3V_n/2 - \lfloor \delta n \rfloor })$ in total variation does \emph{not} hold, because the mass of a linear number of components places a bias on the mass of the $3$-connected core. However, the proof of Theorem~\ref{te:main3} may  be modified to show that a small $o(n)$ number of components does become independent:

\begin{proposition}
	For any sequence of integer $m_n = o(n)$ we have
\begin{align}
	\label{eq:smallindep}
	d_{\mathrm{TV}}\left( (X_n, (\cD_i(\mC_n))_{1\le i \le \min(m_n, V_n) }), (X_n, (\mD(i))_{1\le i \le m_n} ) \right) \to 0.
\end{align}
\end{proposition}
\begin{proof}
	For all $M_1, M_2>0$ we define the collection $\cE_{n, M_1, M_2}$ of finite sequences \[
	E = (3k, y_1, \ldots, y_{m_n})
	\]  with  integers $k, y_1, y_2, \ldots \ge 0$ satisfying the following properties:
	\begin{enumerate}[\qquad a)]
		\item $\Pr{X = 3k}>0$  and $\Pr{Y=y_i} >0$ for all  $1 \le i \le m_n$,
		\item $|\sum_{i=1}^{m_n} y_i - m_n \Ex{Y} | \le  M n^{2/3}$,
		\item $|k(3 \Ex{Y} + 2) - n| \le   M_2  n^{2/3}$.
	\end{enumerate}
	For any such sequence $E$ we set $\ell = \sum_{i=1}^{m_n} y_i$. Note that $\ell = o(n)$ uniformly for all $E \in \cE_{n, M_1, M_2}$.  By Equations~\eqref{eq:zy} and~\eqref{eq:llty}, it holds uniformly for $E \in \cE_{n, M_1, M_2}$
	\begin{align}
		\label{eq:quotient22}
		\frac{\Prb{ (X_n, (Z_i)_{1\le i \le  m_n })  = E } } {\Prb{ (X_n, (Y_i)_{1\le i \le m_n })  = E  }}		
		&= \frac{\Prb{ \ell + \sum_{i=3k- m_n}^{3k} Y_i = n-2k }}{ \Prb{\sum_{i=1}^{3k}  Y_i = n - 2k}  } \\
		&=  \frac{\Prb{ \sum_{i=1}^{3k - m_n } Y_i = n-2k-\ell }}{ \Prb{\sum_{i=1}^{3k}  Y_i = n - 2k}  }  \nonumber\\
		&\to 1 \nonumber
	\end{align}
as $n \in 2 \ndN$ tends to infinity. By Theorem~\ref{te:main2} and~\eqref{eq:llty} it follows that for any $\epsilon>0$ we may choose $M_1, M_2>0$ large enough so that
\[
	\Pr{ (X_n, (Y_i)_{1\le i \le m_n })  \in \cE_{n, M_1, M_2} } > 1- \epsilon/2
\]
for all sufficiently large $n$. By~\eqref{eq:quotient22} it follows that
\[
\Pr{ (X_n, (Z_i)_{1\le i \le m_n })  \in \cE_{n, M_1, M_2} } \sim \Pr{ (X_n, (Y_i)_{1\le i \le m_n })  \in \cE_{n, M_1, M_2} } 
\]
and hence
\[
 \Pr{ (X_n, (Y_i)_{1\le i \le m_n })  \in \cE_{n, M_1, M_2} } > 1 -\epsilon
\]
for all sufficiently large $n$. This completes the proof.
\end{proof}

Furthermore, we expect  the following approximation to hold:

\begin{conjecture}
	Let $0< \delta < 3 \kappa / 2$ be given. As $n \in 2\ndN$ tends to infinity, 
	\begin{align}
		\label{eq:conj}
		d_{\mathrm{TV}}( (\cD_i(\mC_n))_{1\le i \le \min(\lfloor \delta n \rfloor, V_n) }, (\mD(i))_{1\le i \le \lfloor \delta n \rfloor } ) \to 0.
	\end{align}
\end{conjecture}
That is, we conjecture  asymptotic independence of the components \emph{from each other}, but not from the core. It appears that a decent amount of work would be needed to verify~\eqref{eq:conj}. We do not require this result here,  hence we leave it as an open problem. 

\subsection{Local convergence of cubic planar graphs}
\label{sec:endgame}

Let us consider the random graph $\mO_n$ obtained  by removing all components $(\cD_i(\mC_n))_{1 \le i \le 3 V_n / 2}$ from $\mC_n$ and inserting at each edge of the $3$-connected core $\cM(\mC_n)$ an independent copy of the Boltzmann network $\mD$. Thus, letting $\mD(1), \mD(2), \ldots$, denote independent copies of $\mD$, we set $\cM(\mO_n) = \cM(\mC_n)$ and $\cD_i(\mO_n) = \mD(i)$ for $1 \le i \le 3 V_n/2$.

 Clearly the distribution of  $\mO_n$ differs from the distribution of $\mC_n$. For instance, $\mC_n$ has precisely $n$ vertices, whereas  $\mO_n$ has a random number of vertices that is distributed like $V_n + \sum_{i=1}^{3V_n/2} Y_i$. That is, by~\eqref{eq:llty} and Theorem~\ref{te:main2}, it concentrates at $n$, but exhibits Airy-type fluctuations (obtained by concatenating two $3/2$ stable laws) of order $n^{2/3}$. However, as we shall prove, local convergence of $\mO_n$ implies local convergence of $\mC_n$. 

Our first step will be to verify local convergence of the modified model $\mO_n$. For cubic graphs (but not for non-regular graphs), a uniform random vertex may be distinguished by first selecting a uniform random edge and then choosing one of its ends according to a fair independent coin flip. We will hence focus on vicinities of random edges from now on. 

The advantage is that we may partition the edge set of a cubic planar graph according to the components inserted at the $3$-connected core. Recall  that, as illustrated in Figure~\ref{fi:poly}, inserting a $(1 + \cD)$-structure $D$ at an edge $e$ of a $3$-connected network means that either we leave $e$ unchanged if $D$ has size zero, or we  delete $e$ and the root edge of $D$ and  connect its endpoints to the poles of $D$ via two new edges $e_1$ and $e_2$, if $D$ has positive size. Thus, we define the non-root edges of $D$ and either $e$  (if $D$ has size $0$), or $e_1$ and $e_2$ (if $D$ has positive size) to be \emph{associated} to $D$, then each edge of a cubic planar graph $C$ with a unique largest $3$-connected component $M$ is associated to a unique component $\cD_i(C)$ for some integer index $i$ in the range from $1$ to the number of vertices of~$M$.

The number $W$ of edges that we associate to the Boltzmann network $\mD$ has probability generating function
\begin{align}
	\Ex{w^W} = \Pr{Y=0} w  + \sum_{\substack{k \ge 4 \\ k \in 2 \ndN} } \Pr{Y=k} w^{3k/2 + 1}.
\end{align}
We define the size-biased version $\hat{W}$ of $W$ with distribution determined by
\begin{align}
	\Pr{\hat{W}=k} = \frac{k \Pr{W=k}}{\Ex{W}}.
\end{align}
\begin{definition}[Size-biased Boltzmann network]
	\label{def:sizebiased}
We let $\hat{\mD}$ denote a uniformly selected $(1 + \cD)$-structure that is associated to $\hat{W}$ edges. That is, for $\hat{W}=1$, $\hat{\mD}$ equals a placeholder value representing that inserting $\hat{\mD}$ at an edge of a $3$-connected cubic planar graphs leaves the graph unchanged. For $\hat{W} \in \{7, 9, \ldots\}$, $\hat{\mD}$ is uniformly selected among all $\cD$-networks with $\hat{W}-1$ edges. We distinguish a uniformly selected edge associated to $\hat{\mD}$. 
\end{definition}

\begin{lemma}
	\label{le:sizebias}
Let $r \ge 1$ denote a fixed integer. Let $e_1, \ldots, e_r$ denote uniformly selected edges of the random  graph $\mO_n$. For each $1 \le i \le 3 V_n / 2$ we let $\tilde{\cD}_i(\mO_n)$ denote $\cD_i(\mO_n)$ with the additional information which edges associated to $\cD_i(\mO_n)$ coincide with which edge from $e_1,\ldots, e_r$. 
Let $j_1, \ldots, j_r$ denote uniformly selected distinct elements of $\{1, \ldots, 3V_n/2\}$. Let $\hat{\mD}(1), \hat{\mD}(2), \ldots$ denote independent copies of $\hat{\mD}$. For each $1 \le i \le 3 V_n / 2$ we set
\begin{align}
	\tilde{\mD}(i) = \begin{cases} \mD(i), &i \notin \{j_1, \ldots, j_r\} \\ \hat{\mD}(i), &i=j_k\text{ with } 1 \le k \le r. \end{cases}
\end{align}
Then
\begin{align}
	\label{eq:approxresult}
	d_{\mathrm{TV}}\left( (\tilde{\cD}_i(\mO_n))_{1 \le i \le 3 V_n /2} , (\tilde{\mD}(i))_{1 \le i \le 3 V_n /2} \right) \to 0
\end{align}
as $n \in 2 \ndN$ tends to infinity.
\end{lemma}
\begin{proof}
	With high probability all $r$ marked edges are going to be part of distinct components. We set $L := 3 V_n / 2$. Let $(D_1, \ldots, D_L)$ be a sequence of $(1 + \cD)$-structures, among which precisely $r$ members $D_{\ell_1}, \ldots, D_{\ell_r}$ carry a marked associated edge. Let $a_i$ denote the number of associated edges of $D_i$ for all $1 \le i \le L$.  That is, $a_i$ equals one plus the number of edges of $D_i$. Then
	\begin{align*}
		\frac{\Prb{(\tilde{\cD}_i(\mO_n))_{1 \le i \le \ell} = (D_1, \ldots, D_L) \mid L } }{ \Prb{(\tilde{\mD}(i))_{1 \le i \le \ell} = (D_1, \ldots, D_L) \mid L } } = \prod_{j=1}^r \frac{(L -j+1)\Ex{W}}{\sum_{i=1}^r a_i - \sum_{m=1}^{j-1} a_{\ell_m}}.
	\end{align*}
	The limit~\eqref{eq:approxresult} now follows from the law of large numbers.
\end{proof}

\begin{lemma}
	\label{le:convo}
	Let $v_n$ denote a uniformly selected vertex of $\mO_n$. 
	There is a random infinite cubic planar graph $\hat{\mC}$ such that 
	\begin{align}
		\label{eq:quenchedo}
		\mathfrak{L}((\mO_n, v_n) \mid \mO_n) \convdis \mathfrak{L}(\hat{\mC}).
	\end{align}
\end{lemma}
\begin{proof}
	We let $e_1$ and $e_2$ denote independent uniform random edges of $\mO_n$. Let $1 \le i_1, i_2 \le 3V_n / 2$ denote the unique indices of the network components of $\mO_n$ to which $e_1$ and $e_2$ are associated. It is important to note that conditional on $\mO_n$, the indices $i_1$ and $i_2$ are not uniformly distributed, because they are more likely to be associated to large components than to small. However, the distribution of $i_1$ and $i_2$ conditional only on $V_n$  is that of uniformly and independently selected integers from $\{1, \ldots, 3 V_n / 2\}$. Let $f_1$ and $f_2$ denote the corners of  the $3$-connected core $\cM(\mC_n)$ obtained from the edges corresponding to $i_1$ and $i_2$ by orienting them according to fair independent coin flips. It follows by Lemma~\ref{le:conv3cubic} that 
	\begin{align}
	\label{eq:convergence}
	((\cM(\mC_n), f_1), (\cM(\mC_n),f_2)) \convd (\hat{\mM}^{(1)},\hat{\mM}^{(2)})
	\end{align}
	with $\hat{\mM}^{(1)},\hat{\mM}^{(2)}$ denoting independent copies of $\hat{\mM}$.
	
	 Let $r \ge 1$ denote a fixed integer. Each edge of $\mO_n$ corresponds to a unique edge of $\cM(\mC_n)$. The $r$-neighbourhoods of $e_1$ and $e_2$ in $\mO_n$ hence correspond to collections $I_1$ and $I_2$ of edges of $\cM(\mC_n)$. Since, for $i=1,2$, the set $I_i$ is a subset of the collection of edges $J_i$ of the $r$-neighbourhood of $f_i$ in $\cM(\mC_n)$, it follows that 
	 \begin{align}
	 	\label{eq:i1i2disjoint}
	 	\Pr{I_1 \cap I_2 = \emptyset} \to 1
	 \end{align}
	 as $n \in 2\ndN$ tends to infinity. Indeed, local convergence of $(\cM(\mC_n), f_1)$ entails that the $2r$-neighbourhood of $f_1$ has stochastically bounded number of edges. Among the $3V_n/2 = (3/2) \kappa n + O_p(n^{2/3})$ edges, it is unlikely that the uniformly distributed edge  $f_2$ falls into this stochastically bounded set. Hence, $f_1$ and $f_2$ have distance at least $2r$ with probability tending to $1$ as $n \in 2 \ndN$ tends to infinity. Hence
	 \[
	 	\Pr{J_1 \cap J_2 = \emptyset} \to 1,
	 \]
	 and this verifies~\eqref{eq:i1i2disjoint}.

	 By Lemma~\ref{le:sizebias} and the limit \eqref{eq:i1i2disjoint}, it follows that jointly the components corresponding to $i_1$ and $i_2$ behave like independent copies of $\hat{\mD}$, the components corresponding to all edges from $(I_1 \cup I_2) \setminus \{i_1, i_2\}$ behave like independent copies of $\mD$, and the $r$-neighbourhoods of $i_1$ and $i_2$ in $\cM(\mC_n)$ like  independent copies of the $r$-neighbourhood of the root in $\hat{\mM}$.  Since the $r$-neighbourhoods of $e_1$ and $e_2$ in $\mO_n$ are entirely determined by these pieces, it follows that
	 \begin{align}
	 		((\mO_n, e_1), (\mO_n,e_2)) \convd (\hat{\mC}^{(1)},\hat{\mC}^{(2)}),
	 \end{align}
 	with $\hat{\mC}^{(1)}, \hat{\mC}^{(2)}$ denoting independent copies of a random infinite cubic planar graph $\hat{\mC}$ constructed by inserting  independent copies of $\mD$ into each non-root edge of $\hat{\mM}$, and inserting an independent copy of $\hat{\mD}$ at the root edge of $\hat{\mM}$. This implies~\eqref{eq:quenchedo} if we view $\hat{\mC}$ as vertex-rooted instead, by distinguishing an endpoint of its root edge according to an independent fair coin flip.
\end{proof}

We are now ready to prove our main result, the local convergence of $\mC_n$. Having established local convergence of the modified model $\mO_n$ towards a limit $\hat{\mC}$, it remains to argue that $\mO_n$ is a good approximation of $\mC_n$. We will do so using Theorem~\ref{te:main3}.

\begin{proof}[Proof of Theorem~\ref{te:main}]
	By Markov's inequality it follows that for any $0 < \delta < 3 \kappa/2$ the sum of the number of vertices $v(\mD(i))$ incident to edges associated to the network component $\mD(i)$, $1 \le i \le 3 V_n /2$, satisfies
	\begin{align}
		\label{eq:indstep}
		\sum_{i= 3V_n/2 - \lfloor \delta n \rfloor +1}^{3 V_n / 2} v(\mD(i)) \le 2\Ex{W} \delta n
	\end{align}
with probability tending to $1$ as $n \in 2\ndN$ tends to infinity.

We let $F_{n, \delta}$ denote the collection of vertices of $\mO_n$ that are incident to edges associated to components $\cD_i(\mO_n) = \mD(i)$ for $3 V_n / 2 - \lfloor \delta n \rfloor < i \le 3 V_n / 2$. For each integer $r \ge 0$ we let $U_r(F_{n, \delta})$ denote the collection of vertices in $\mO_n$ with graph distance at most $r$ from $F_{n, \delta}$. We will show that for each $r \ge 0$ and $\epsilon'>0$
\begin{align}
	\label{eq:ulem}
	\lim_{\delta \downarrow 0} \limsup_{n \to \infty} \Pr{ |U_r(F_{n,\delta})| > \epsilon' n} = 0. 
\end{align}

In order to prove~\eqref{eq:ulem} we proceed by induction on $r$.
For $r = 0$,~\eqref{eq:ulem} follows directly from Inequality~\eqref{eq:indstep}. For the induction step, let  $r \ge 1$ and assume that the statement holds for $r-1$. If~\eqref{eq:ulem} fails for $r$, then there exists a constant $\epsilon_0 >0$, a sequence $(n_\ell)_{\ell \ge 1}$ of even integers, and sequences $\epsilon'_\ell, \delta_\ell \downarrow 0$ (as $\ell \to \infty$) such that
\[
\Pr{|U_r(F_{n_\ell, \delta_\ell})| > \epsilon' n_\ell, |U_{r-1}( F_{n_\ell, \delta_\ell})| < \epsilon'_\ell n_\ell} > \epsilon_0
\]
for all integers $\ell \ge 1$. 

Hence with probability at least $\epsilon_0>0$ the following statements hold: There are at least $n_\ell( \epsilon' - \epsilon'_\ell)$ vertices in $U_{r}(F_{n_\ell, \delta_\ell}) \setminus U_{r-1}(F_{n_\ell, \delta_\ell})$.  Each of these vertices is incident to an edge whose other end belongs to $U_{r-1}(F_{n_\ell, \delta_\ell})$. We choose a sequence $(d_\ell)_\ell$ such that $d_\ell \to \infty$ and $d_\ell \epsilon_\ell' \to 0$. The subset of vertices in $U_{r-1}(F_{n_\ell, \delta_\ell})$ with degree smaller than $d_\ell$ is linked to at most $d_\ell \epsilon'_\ell n_\ell = o(n_\ell)$ vertices from $U_{r}(F_{n_\ell, \delta_\ell}) \setminus U_{r-1}(F_{n_\ell, \delta_\ell})$. It follows that at least $(\epsilon' - o(1))n_\ell$ vertices from $U_{r}(F_{n_\ell, \delta_\ell}) \setminus U_{r-1}(F_{n_\ell, \delta_\ell})$ have a neighbour from $U_{r-1}(F_{n_\ell, \delta_\ell})$ with degree at least $d_\ell$.

Thus, with probability at least $\epsilon_0>0$, a uniformly selected vertex of $\mO_n$ has a neighbour with degree at least $d_\ell$. As $d_\ell \to \infty$, this contradicts Lemma~\ref{le:convo}, by which $\mO_n$ rooted at a uniform random vertex has a local limit and hence the $2$-neighbourhood of that vertex needs to have a stochastically bounded size. This completes the induction and hence the verification of~\eqref{eq:ulem}.

Let $G$ be an arbitrary rooted connected graph with radius at most $r$. We let $N_G(\mO_n)$ and $N_G(\mC_n)$ denote the number of vertices in $\mO_n$ and $\mC_n$ whose $r$-neighbourhood is isomorphic to $G$ as rooted graphs. By Lemma~\ref{le:convo},
\[
	\frac{N_G(\mO_n)}{|\mO_n|} \convp \Pr{U_r(\hat{\mC}) \simeq G}.
\]
For ease of notation, we set $p_G := \Pr{U_r(\hat{\mC}) \simeq G}$.
By Equation~\eqref{eq:llty} and Theorem~\ref{te:main2},
\[
	n^{-1}|\mO_n| = n^{-1}(V_n + \sum_{i=1}^{3V_n/2} Y_i) \convp 1.
\]
Hence
\begin{align}
	\label{eq:concentratep}
	\frac{N_G(\mO_n)}{n} \convp p_G.
\end{align}

Now, let $\epsilon,\epsilon'>0$ be given. By Equation~\eqref{eq:ulem} we may choose $\delta>0$ small enough such that for all sufficiently large $n$
\begin{align}
	\label{eq:upupup}
	\Pr{ |U_r(F_{n,\delta})| > \epsilon' n} < \epsilon.
\end{align}
By Theorem~\ref{te:main3} there exist constants $0<c<C$ and sets $(\cE_n)_{n \ge N}$ such that for all sufficiently large $n$ Inequalities~\eqref{eq:doa},~\eqref{eq:dob}, and~\eqref{eq:N2} hold uniformly for all $E \in \cE_n$. We define a subset $\cE_n(1) \subset \cE_n$ of sequences $(M, D_1, D_2, \ldots)$ with the property that there are at most $\epsilon' n$ vertices in $M$ with graph distance at most $r$ from the last $\lfloor \delta n \rfloor$ edges of $M$. By Inequality~\eqref{eq:upupup} we know that $\cM(\mC_n)$ has this property with probability at least $1 - \epsilon$. It follows that
\begin{align}
	\label{eq:firstpart}
	\Prb{ (\cM(\mC_n), (\cD_i(\mC_n))_{1\le i \le 3V_n/2 - \lfloor \delta n \rfloor })  \notin \cE_n(1)  } < 2\epsilon
\end{align}
and
\begin{align}
	\Prb{ (\cM(\mC_n),  (\mD(i))_{1\le i \le 3V_n/2 - \lfloor \delta n \rfloor }) \notin \cE_n(1)  } < 2\epsilon
\end{align}
and for all elements $E \in \cE_n(1)$
\begin{align}
	\label{eq:notin}
	c < \frac{\Prb{ (\cM(\mC_n), (\cD_i(\mC_n))_{1\le i \le 3V_n/2 - \lfloor \delta n \rfloor })  =E  }} {	\Prb{ (\cM(\mC_n),  (\mD(i))_{1\le i \le 3V_n/2 - \lfloor \delta n \rfloor }) =E  }}  < C.
\end{align}
We may write $N_G(\cdot) = N_G'(\cdot) + N_G''(\cdot)$ with $N_G'(\cdot)$ counting only the vertices whose $r$-neighbourhood is isomorphic to $G$ and does not contain any vertex that is incident to any of the last $\lfloor \delta n \rfloor$ edges of $\cM(\mC_n)$. 
Note that $N_G'(\cdot)$ is  entirely determined by $\cM(\mC_n)$ and all but the last $\lfloor \delta n \rfloor$ network components attached to $\cM(\mC_n)$.  This allows us to define the subset $\cE_n(2) \subset \cE_n(1)$ of all configurations in $\cE_n(1)$ for which \[
|N'_G(\cdot) / n - p_G| > 2 \alpha \epsilon'
\] for 
\[
	\alpha := 7 \Ex{W}.
\]
Not all  configurations with this inequality property lie in~$\cE_n(1)$, but $\cE_n(2)$ really only contains those that do. By the triangle inequality, $|N'_G(\cdot) / n - p_G| > 2 \alpha \epsilon'$ entails that $|N_G(\cdot) / n - p_G| >  \alpha \epsilon'$ or $N_G''(\cdot) / n > \alpha \epsilon'$. Hence, by Inequality~\eqref{eq:notin},
\begin{align}
	\label{eq:masterineq}
	&\Prb{ (\cM(\mC_n), (\cD_i(\mC_n))_{1\le i \le 3V_n/2 - \lfloor \delta n \rfloor })  \in \cE_n(2)  } \\
	&\le C \Prb{ (\cM(\mC_n), (\mD(i))_{1\le i \le 3V_n/2 - \lfloor \delta n \rfloor })  \in \cE_n(2)  } \nonumber \\
	&\le C \Prb{ |N_G(\mO_n)/n - p_G| > \alpha\epsilon'}  \nonumber \\
	&\quad + C \Prb{ N_G''(\mO_n)/n > \alpha\epsilon', (\cM(\mC_n), (\mD(i))_{1\le i \le 3V_n/2 - \lfloor \delta n \rfloor })  \in \cE_n(2)}. \nonumber
\end{align}
By the limit~\eqref{eq:concentratep} we know 
\begin{align}
	\label{eq:chief1}
	\Prb{ |N_G(\mO_n)/n - p_G| > \alpha\epsilon'} = o(1).
\end{align}
As for the second summand in the upper bound, since $\cE_n(2) \subset \cE_n(1)$ we know that for each configuration in $\cE_n(2)$ there are  at most $\epsilon'n$ vertices  in the $3$-connected core whose graph distance is at most $k$ from the last $\lfloor \delta n \rfloor$ edges. Hence the collection $I$ of indices of edges that are incident to  such a vertex and that do not belong to the last $\lfloor \delta n \rfloor$ satisfy
\[
	|I| + \lfloor \delta n \rfloor \le 3 \epsilon' n.
\]
Setting $J = I \cup \{3 V_n / 2- \lfloor \delta n \rfloor+1,\ldots,  3 V_n/2\}$, it follows by Markov's inequality that
\begin{align*}
	\sum_{i\in J} v(\mD(i)) \le 6\Ex{W}  \epsilon' n
\end{align*}
with probability tending to $1$ as $n \to \infty$.  However, if $N_G''(\mO_n) / n > \alpha\epsilon'n$, then $\sum_{i\in J} v(\mD(i)) > \alpha\epsilon'n$. Thus,
\begin{align}
	\label{eq:chief2}
	\Prb{ N_G''(\mO_n)/n > \alpha\epsilon', (\cM(\mC_n), (\mD(i))_{1\le i \le 3V_n/2 - \lfloor \delta n \rfloor })  \in \cE_n(2)} = o(1).
\end{align} 
Combining~\eqref{eq:masterineq},~\eqref{eq:chief1}, and~\eqref{eq:chief2}, it follows that
\begin{align}
	\Prb{ (\cM(\mC_n), (\cD_i(\mC_n))_{1\le i \le 3V_n/2 - \lfloor \delta n \rfloor })  \in \cE_n(2)  } \to 0.
\end{align}
By Inequality~\eqref{eq:firstpart} it follows that
\begin{align}
	\Pr{|N_G'(\mC_n)/ n - p_G| > \alpha \epsilon'} < 3 \epsilon
\end{align}
for all sufficiently large $n$. Since $\epsilon, \epsilon'>0$ were arbitrary, it follows that
\begin{align}
	\label{eq:stochlower}
	N_G(\mC_n)/n - p_G \ge o_p(1).
\end{align}
In order to verify a matching upper bound, we select a finite set $\mathfrak{G}_0$ of unlabelled rooted graphs with $G \in \mathfrak{G}_0$ and \[
\sum_{H \in \mathfrak{G}_0} p_H = \Pr{U_r(\hat{\mC}) \in \mathfrak{G}_0 } > 1 - \epsilon.
\]
Clearly it holds that
\[
	\sum_{H \in \mathfrak{G}_0} N_H(\mC_n) / n \le 1
\]
and consequently 
\[
 \sum_{H \in \mathfrak{G}_0}(N_H(\mC_n)/n - p_H) \le \epsilon.
\]
Since $\mathfrak{G}_0$  is finite and since Inequality~\eqref{eq:stochlower} holds for arbitrary rooted graphs $G$, it follows that there is a (possibly negative) sequence $(t_n)_{n \ge 1}$ with $t_n = o(1)$ such that 
\[
	\Pr{ N_H(\mC_n) / n - p_H \ge t_n \text{ for all $H \in \mathfrak{G}$}} \to 1
\]
as $n \in 2 \ndN$ tends to infinity. As $G \in \mathfrak{G}_0$, it follows that
\[
	N_G(\mC_n) / n - p_G \le 2\epsilon
\]
with probability tending to $1$ as $n$ becomes large. As $\epsilon>0$ was arbitrary, it follows with~\eqref{eq:stochlower}
\[
	N_G(\mC_n)/n \convp p_G.
\]
As this holds for arbitrary $G$, it follows that
\begin{align*}
	\mathfrak{L}((\mC_n, v_n) \mid \mC_n) \convdis \mathfrak{L}(\hat{\mC}),
\end{align*}
with $v_n$ denoting a uniformly selected vertex of $\mC_n$.
\end{proof}

\section{Disconnected cubic planar graphs}
\label{sec:cubicdisconnected}

We also obtain local convergence of the  uniform random  $n$-vertex simple cubic planar graph $\mG_n$ that is not required to be connected. 

\begin{corollary}
	\label{co:concon}
Let $u_n$ denote a uniformly selected vertex of $\mG_n$. Then
\begin{align}
	\mathfrak{L}((\mG_n, u_n) \mid \mG_n) \convdis \mathfrak{L}(\hat{\mC}).
\end{align}
\end{corollary}
Corollary~\ref{co:concon} follows from~Theorem~\ref{te:main}, because $\mG_n$ exhibits a unique giant connected component, and the mass of the small fragments $\mathrm{frag}(\mG_n)$ obtained by deleting the largest component remains stochastically bounded. The small fragments even exhibit a so-called Boltzmann--Poisson random graph as limit:
\begin{proposition}
	\label{pro:connected}
	The largest connected component of the random $n$-vertex cubic planar graph $\mG_n$ has $n - O_p(1)$ vertices. As unlabelled finite random graphs, it holds that
	\[
	\mathrm{frag}(\mG_n) \convdis \mG
	\]
	for a Boltzmann--Poisson random graph $\mG$ that assumes any finite unlabelled cubic planar graph $G$ with $|G|$ vertices and $\mathrm{aut}(G)$ automorphisms with probability
	\[
	\Pr{\mG = G} = \frac{\rho^{|G|}}{\mathrm{aut}(G) \exp(\cC(\rho))}
	\]
	for constants $\rho \approx 0.319224$ and $\cC(\rho) \approx 0.00060$ defined in Section~\ref{sec:largest}.
\end{proposition}
Proposition~\ref{pro:connected} readily follows from a general result on random set partitions~\cite[Thm. 3.4]{doi:10.1002/rsa.20771} and the asymptotic growth formula for  the number of connected cubic planar graphs by~\cite{zbMATH05122852, zbMATH07213288}. Such a behaviour is quite universal for random graphs from restricted classes, see for example~\cite[Thm. 1.3]{MR2507738} and~\cite[Thm. 4.2]{doi:10.1002/rsa.20771}. The number of connected components of $\mG$ follows a Poisson distribution, see also~\cite[Thm. 4.6]{MR3068033} and~\cite[Thm. 2]{zbMATH07213288}, hence the name Boltzmann--Poisson random graph.

\section*{Acknowledgement}

I am grateful to  Michael Drmota for related discussions.

\bibliographystyle{abbrv}
\bibliography{cubic}

\end{document}